\documentclass[11pt,reqno]{amsart}
\usepackage[colorlinks=true,allcolors=blue,backref=page]{hyperref}
\usepackage{color}
\usepackage{amsmath, amssymb, amsthm}
\usepackage{mathrsfs}
\usepackage{mathtools}
\usepackage[noabbrev,capitalize,nameinlink]{cleveref}
\crefname{equation}{}{}
\usepackage{fullpage}
\usepackage[noadjust]{cite}
\usepackage{graphics}
\usepackage{pifont}
\usepackage{tikz}
\usepackage{bbm}
\usepackage[T1]{fontenc}

\usetikzlibrary{arrows.meta}

\usepackage{environ}
\usepackage{framed}
\usepackage{url}
\usepackage[linesnumbered,ruled,vlined]{algorithm2e}
\usepackage[noend]{algpseudocode}
\usepackage[labelfont=bf]{caption}
\usepackage{cite}
\usepackage{framed}
\usepackage[framemethod=tikz]{mdframed}
\usepackage{appendix}
\usepackage{graphicx}
\usepackage[textsize=tiny]{todonotes}
\usepackage{tcolorbox}
\usepackage{enumerate}
\usepackage{verbatim}
\usepackage[shortlabels]{enumitem}
\allowdisplaybreaks[1]

\apptocmd{\sloppy}{\hbadness 10000\relax}{}{} % magic 
\allowdisplaybreaks[1]

\usepackage[color,final]{showkeys} %add in 'final' into parameter to remove showkeys

\colorlet{refkey}{orange!20}
\colorlet{labelkey}{blue!30}

\crefname{algocf}{Algorithm}{Algorithms}

\let\originalleft\left
\let\originalright\right
\renewcommand{\left}{\mathopen{}\mathclose\bgroup\originalleft}
\renewcommand{\right}{\aftergroup\egroup\originalright}

\numberwithin{equation}{section}
\newtheorem{theorem}{Theorem}[section]
\newtheorem{proposition}[theorem]{Proposition}
\newtheorem{lemma}[theorem]{Lemma}

\crefname{claim}{Claim}{Claims}

\newtheorem*{question*}{Question}

\theoremstyle{definition}
\newtheorem{definition}[theorem]{Definition}

\newtheorem*{definition*}{Definition}

\theoremstyle{remark}

\usepackage{appendix}
\AtBeginEnvironment{appendices}{\crefalias{section}{appendix}} %appendices
\crefname{appsec}{Appendix}{Appendices}
\crefname{equation}{}{} %remove ``Equation''
\crefformat{enumi}{#2#1#3}
\crefrangeformat{enumi}{#3#1#4 to~#5#2#6}
\crefmultiformat{enumi}{#2#1#3}
{ and~#2#1#3}{, #2#1#3}{ and~#2#1#3}

\newcommand{\mb}{\mathbb}
\newcommand{\mbf}{\mathbf}
\newcommand{\mbm}{\mathbbm}
\newcommand{\mc}{\mathcal}
\newcommand{\mf}{\mathfrak}
\newcommand{\mr}{\mathrm}

\newcommand{\ol}{\overline}
\newcommand{\on}{\operatorname}
\newcommand{\wh}{\widehat}

\global\long\def\GG{\mathbb{G}}
\global\long\def\E{\mathbb{E}}
\global\long\def\Var{\operatorname{Var}}
\global\long\def\vol{\operatorname{vol}}
\global\long\def\NN{\mathbb{N}}
\global\long\def\RR{\mathbb{R}}
\global\long\def\Pr{\mathbb{P}}
\global\long\def\eps{\varepsilon}

\title{Friendly bisections of random graphs}

\author[Ferber]{Asaf Ferber}

\address{Department of Mathematics, University of California, Irvine.}

\email{asaff@uci.edu}

\author[Kwan]{Matthew Kwan}

\address{Department of Mathematics, Stanford University, Stanford, CA.}

\email{mattkwan@stanford.edu}

\author[Narayanan]{Bhargav Narayanan}

\address{Department of Mathematics, Rutgers University, Piscataway, NJ 08854, USA}

\email{narayanan@math.rutgers.edu}

\author[Sah]{Ashwin Sah}
\author[Sawhney]{Mehtaab Sawhney}

\address{Department of Mathematics, Massachusetts Institute of Technology, Cambridge, MA 02139, USA}
\email{\{asah,msawhney\}@mit.edu}

\date{\today}

\begin{document}

\begin{abstract}
Resolving a conjecture of F\"uredi from 1988, we prove that with high probability, the random graph $\GG(n,1/2)$ admits a friendly bisection of its vertex set, i.e., a partition of its vertex set into two parts whose sizes differ by at most one in which $n-o(n)$ vertices have at least as many neighbours in their own part as across. Our proof is constructive, and in the process, we develop a new method to study stochastic processes driven by degree information in random graphs; this involves combining enumeration techniques with an abstract second moment argument.
\end{abstract}

\maketitle

\section{Introduction}

In a cut of a graph, i.e., a partition of its vertex set into two parts, we call a vertex \emph{friendly} if it has more neighbours in its own part than across, and \emph{unfriendly} otherwise. Questions about finding friendly and unfriendly partitions of graphs, i.e., partitions in which all (or almost all) the vertices are friendly or unfriendly, have been investigated in various contexts: in combinatorics, on account of their inherent interest~\cite{T83,S96,BS02,BL16,GK00,SD02,LL20}, in computer science, as `local' analogues of important NP-complete partitioning problems~\cite{ABPW17,CGVYZ20}, in probability and statistical physics, owing to their connections to spin glasses~\cite{GL18,GNS18, ADLO19, SGNS20}, and in logic and set theory~\cite{AMP90, SM90}; this list is merely a representative sample (and by no means exhaustive) since such partitions have been studied extremely broadly. On the other hand, when it comes to finding friendly or unfriendly \emph{bisections}, i.e., partitions into two parts whose sizes differ by at most one, much less is known. Our aim here is to prove an old and well-known conjecture about random graphs due to F\"uredi~\cite{F88}. This problem has gained some notoriety over the years, in part due to its inclusion in Green's list of 100 open problems~\cite[Problem~91]{GreOp}. Our main result is as follows.

\begin{theorem}\label{thm:furedi}
With high probability, an Erd\H os--R\'enyi random graph $G\sim\GG(n,1/2)$ admits a bisection in which $n-o(n)$ vertices are friendly.
\end{theorem}

\subsection*{Degree-driven stochastic processes}
Although \cref{thm:furedi} is specifically about friendly bisections of random graphs, the approach we adopt to prove this result is rather general, and it may be that the more important point of this work is its contribution to methodology. Concretely, we develop a method that appears suitable for analysing many different types of stochastic processes on random graphs driven primarily by degree information; for example, in forthcoming work, the fourth and fifth authors~\cite{MAJDYN} use modifications of these techniques to settle various conjectures of Tran and Vu~\cite{TV20} concerning majority dynamics on random graphs. Below, we outline how our approach allows us to prove \cref{thm:furedi}.

We adopt a constructive approach that yields an efficient algorithm to find the bisection promised by \cref{thm:furedi}. To motivate our approach, it is instructive to consider the following basic algorithm, motivated by the classical large-cut-finding algorithm: starting with any bisection $A\cup B$ of a graph $G$, repeatedly check whether there are vertices $v\in A$ and $w\in B$ such that $\deg_B(v)>\deg_A(v)$ and $\deg_A(w)>\deg_B(v)$, and if so, swap $v$ and $w$. It is easy to see that such a swap must decrease the size (i.e., the number of crossing edges) of the bisection, so this algorithm must terminate. Of course, if we are unlucky, it might happen that when the algorithm terminates, all the vertices in $A$ are friendly, while very few of the vertices in $B$ are friendly, so the resulting bisection may be very far from satisfying the conclusion of \cref{thm:furedi}. However, it seems plausible that such an outcome is rather unusual: if $G$ is sampled from $\GG(n,1/2)$, then one might expect this algorithm (interpreted as a random process) to typically follow a predictable trajectory, and in particular, the number of friendly vertices in $A$ and in $B$ to stay roughly the same for most of the duration of the algorithm.

This is a promising starting point, especially due to the fact that we do not actually need to fully understand the typical trajectory of the process. Indeed, we only need to show that at each step $k$, the number of friendly vertices in $A$ concentrates around some value $N_k$. By symmetry (assuming for the moment that $n$ is even), the number of friendly vertices in $B$ would then concentrate around $N_k$ as well, so the numbers of friendly vertices in $A$ and $B$ would never get `too imbalanced'. However, it is far from obvious how to actually establish concentration. Roughly speaking, the main issue is that in order to execute even the first step of the algorithm, we have to inspect every vertex of our graph, meaning that there is seemingly `no remaining randomness' for the second step. This is in contrast with most other random graph processes in the literature (such as $H$-free or $H$-removal processes, as in~\cite{BK10,BFL15,FGM20} for example), where each individual step is defined in terms of a random choice.

There are two ideas that allow us to salvage enough randomness to establish the desired concentration. First, instead of swapping vertices one at at time, we shall instead swap a sizeable `batch' of vertices between $A$ and $B$ in each step; this is strongly reminiscent of the influential `nibbling' idea introduced by R\"odl~\cite{Rod85}. We will be able to use discrepancy properties of random graphs to show that, in a typical outcome of the random graph $\GG(n,1/2)$, when we have a bisection $A\cup B$ in which many vertices in $A$ and in $B$ are unfriendly, swapping a large number of the `unfriendliest' vertices in $A$ and in $B$ dramatically decreases the size of the bisection. That is to say, it should only take a few steps, (about $\exp(1/\varepsilon)$, in fact) to reach a bisection in which one of the two parts has $(1-\varepsilon)n/2$ friendly vertices. This makes the problem of establishing concentration more tractable, since we now only need to do this for a large constant number of steps. Our second main observation is that in order to execute a step of our algorithm, we only need to know the degrees $\deg_A(v)$ and $\deg_B(v)$ for each vertex $v$ at that stage (and not any other information about the graph). Thus, instead of revealing the whole graph to study the first step, we may simply reveal the required degree information, meaning that our random graph is now conditionally a degree-constrained random graph. We then have the randomness of this degree-constrained random graph with which to show concentration at the next step, for which we again only need to (dynamically) reveal some more degree information, and so on.

The above observations leave us with the task of demonstrating concentration in some (families of) degree-constrained random graphs. In order to study these degree-constrained random graphs, we have at our disposal powerful enumeration theorems due to McKay and Wormald~\cite{MW90}, and extensions by Canfield, Greenhill, and McKay~\cite{CGM08}, which give very precise asymptotic formulae for the number of graphs with specified degree information. In principle, this allows one to write down explicit formulae for essentially all relevant probabilities, from which one could attempt to compute the typical trajectory of the process. However, the necessary computations are formidable, and in particular, the various densities under consideration do not appear to have closed-form expressions past the first few iterations.

Our approach to circumventing these issues brings us to the heart of the matter: we develop an abstract second-moment argument with which one can establish concentration of various statistics at a given step, using only \emph{stability} and \emph{anti-concentration} information about the outcomes of previous steps. In particular, this enables us to establish concentration without actually knowing the trajectory of the process. This is superficially reminiscent of martingale arguments establishing concentration around the mean without any knowledge of the location of the mean itself (see~\cite{AS16}), but the inputs to such arguments, typically Lipschitz-like behaviour of the random variables of interest, are rather different from the inputs to our argument. As mentioned earlier, the methods in our argument are quite general, and we anticipate that a broad range of similar stochastic processes will now become amenable to analysis.

\subsection*{Notation}
Our graph-theoretic notation is for the most part standard; see~\cite{BB98} for terms not defined here. In a graph $G$, we write $\deg(v)$ for the degree of a vertex $v \in V(G)$, and $N(v)$ for its neighbourhood; also, for a subset $U \subseteq V(G)$, we write $\deg_U (v)$ for the number of neighbours of $v$ in $U$, i.e.,  for the size of $N(v) \cap U$.  We write $\GG(n,p)$ for the Erd\H{o}s--R\'enyi random graph on $n$ vertices with edge density $p$. 

Our use of asymptotic notation is mostly standard as well.  We say that an event occurs with high probability if it holds with probability $1 - o(1)$ as some parameter (usually $n$, unless we specify otherwise) grows large. Constants suppressed by asymptotic notation may be absolute, or might depend on other fixed parameters; we shall spell out the latter situation explicitly whenever there might be cause for confusion. To lighten notation, we write $f = g \pm h$ for $|f-g| \le h$. We maintain this convention with asymptotic notation as well, so $f = g \pm n^{-\Omega(1)}$ for example is taken to mean $|f - g| = n^{-\Omega(1)}$. We also adopt the following non-standard bit of notation: as a parameter $n$ grows large, we write $f \simeq h$ if $f = (1 \pm n^{-\Omega(1)})h$. Finally, following a common abuse, we omit floors and ceilings wherever they are not crucial.

\subsection*{Organisation}
This paper is organised as follows. In~\cref{sec:overview}, we describe the swapping process that allows us to prove~\cref{thm:furedi}, and also give the deduction of our main result from a few key lemmas. In~\cref{sec:decrement}, we dispose of the more routine of these lemmas. The beef of our argument is in~\cref{sec:concentration}, where we must work rather hard to establish the key concentration properties of our swapping process.

\section{Proof overview}\label{sec:overview}
In this section we make some initial observations, then describe a random swapping process that underlies our argument and state some facts about this process (with proofs to follow later). We then show how to deduce~\cref{thm:furedi} from these facts.

Given a bisection $A\cup B$ of a graph, the \emph{friendliness} $\Delta_{A,B}(v)$ of a vertex $v$ is the difference between the number of its neighbours on its own side and the number of its neighbours on the other side. We say a vertex is \emph{friendly} if its friendliness is positive, and otherwise, we say it is \emph{unfriendly}. The \emph{total friendliness} $\Delta_{A,B}$ of the bisection  $A \cup B$ is then given by
\[\Delta_{A,B} = \sum_{v \in V(G)} \Delta_{A,B}(v).\]

We also make a simple observation that allows us to restrict our attention to random graphs of even order (which in turn allows us to somewhat simplify the presentation). A simple union bound (similar to calculations we will see in \cref{sec:decrement}) shows that with high probability, in \emph{any} partition of the vertex set of $\GG(n,1/2)$, at most $10n / \log n$ vertices have friendliness $1$, i.e., have exactly one more neighbour on their own side than across, or vice versa. Consequently, it clearly suffices to establish \cref{thm:furedi} for $\GG(n,1/2)$ when $n$ is even; indeed, when $n$ is odd, we may delete an arbitrary vertex from the random graph, apply \cref{thm:furedi} to the result, and add back the deleted vertex to either part to get the desired bisection. Therefore, all graphs under consideration will be of even order unless explicitly specified otherwise, and we shall not belabour this point any further.

The following lemma shows that for a typical outcome of the random graph $\GG(n,1/2)$, there is a window of length $O(n^{3/2})$ within which the total friendliness of any bisection lies.
\begin{lemma} \label{lem:total-friendliness}
There is a  $\gamma>0$ such that for a random graph $G\sim\GG(n,1/2)$, with high probability, every bisection $A\cup B$ of $G$ has $|\Delta_{A,B}|<\gamma n^{3/2}$.
\end{lemma}

Next, we shall define a simple random `swap' operation that modifies a bisection with the aim of making it more friendly.

\begin{definition}
Given a bisection $A\cup B$ of an $n$-vertex graph $G$, the \emph{$\alpha$-swap} of $A\cup B$ is the random bisection obtained by the following procedure. First, we take the subset $A'\subseteq A$ of the $\lfloor\alpha n\rfloor$ most unfriendly vertices in $A$, and the subset $B'\subseteq B$ of the $\lfloor\alpha n\rfloor$ most unfriendly vertices in $B$ (breaking ties according to some \textit{a priori} fixed ordering of the vertex set), and swap $A'$ and $B'$. At this stage, the parts of the resulting bisection are then $(A\setminus A')\cup B'$ and $(B\setminus B')\cup A'$. Next, we make a uniformly random choice of $\lfloor\alpha^4n\rfloor$ vertices on both of these sides, and swap these subsets.
\end{definition}

We remark that the second (random) swap in the $\alpha$-swap procedure is not actually necessary for the proof of \cref{thm:furedi}, but the analysis later in the paper would become substantially more involved without it.

The following lemma shows that in a typical outcome of the random graph $\GG(n,1/2)$, for every bisection $A\cup B$, either our swapping operation increases the total friendliness by $\Omega(n^{3/2})$, or almost all the vertices in one of the parts (either $A$ or $B$) are already friendly.
\begin{lemma}\label{lem:swap-increment}
For every fixed $\varepsilon>0$, there are $\alpha\in(0,\varepsilon)$ and $\beta>0$ for which a random graph $G\sim\GG(n,1/2)$
has, with high probability, the following property. In any bisection $A\cup B$ of $G$ in which at least $\varepsilon n$ vertices are unfriendly in each of $A$ and $B$, the random bisection $A_1\cup B_1$ obtained from an $\alpha$-swap of $A\cup B$ always satisfies
\[\Delta_{A_1,B_1}\ge\Delta_{A,B}+\beta n^{3/2}.\] 
\end{lemma}

Finally, the next lemma establishes concentration properties for bisections obtained by iterating our swapping operation.

\begin{lemma}\label{lem:symmetry}
Fix $\varepsilon>\alpha>0$, $k\in\NN$, and an arbitrary bisection $A\cup B$ of the vertex set of $\GG(n,1/2)$. For a random graph $G\sim\GG(n,1/2)$, let $A_k\cup B_k$ be the bisection obtained by performing $k$ iterations of the $\alpha$-swap procedure starting from $A\cup B$. Writing $X$ and $Y$ respectively for the number of unfriendly vertices in $A_k$ and $B_k$, we have with high probability that~$|X-Y|=o(n)$.
\end{lemma}

With these facts in hand, we may now easily deduce \cref{thm:furedi}.
\begin{proof}[Proof of \cref{thm:furedi}]
For any fixed $\varepsilon>0$, we shall show that $G\sim\GG(n, 1/2)$ with high probability has a bisection in which at most $2\varepsilon n+o(n)$ vertices are unfriendly. 

Say $V(G) = \{1, \dots, n\}$, define the bisection $A_0\cup B_0$ by $A_0=\{1,\dots,n/2\} $
and $B_0=\{n/2+1,\dots,n\}$. Let $\gamma$ be as in \cref{lem:total-friendliness} and $\beta$ as in \cref{lem:swap-increment} applied to $\epsilon$. Set $K=\lceil2\gamma/\beta\rceil+1$, and let 
\[A_1\cup B_1,\;A_2\cup B_2,\dots,A_K\cup B_K\] be the sequence of bisections arising from $K$ iterations of the $\alpha$-swap procedure starting from $A_0 \cup B_0$. 

Say that a bisection $A\cup B$ is \emph{$\varepsilon$-good} if there are at most $\varepsilon n$ unfriendly vertices in $A$ or at most $\varepsilon n$ unfriendly vertices in $B$. Now, the following properties hold with high probability, by \cref{lem:total-friendliness,lem:swap-increment,lem:symmetry}.

\begin{enumerate}
    \item\label{A1} There is an interval of length at most $2\gamma n^{3/2}$ such that the total friendliness of every bisection of $G$ lies in this interval.
    \item\label{A2} For every $0\le k\le K-1$, either $A_k\cup B_k$ is $\varepsilon$-good, or $\Delta_{A_{k+1},B_{k+1}}\ge \Delta_{A_k,B_k}+\beta n^{2/3}$.
    \item\label{A3} For every $1\le k\le K$, the numbers of unfriendly vertices in $A_k$ and in $B_k$ differ by $o(n)$.
\end{enumerate}

Fix outcomes of $G$ and $A_1\cup B_1,A_2\cup B_2,\dots,A_K\cup B_K$ satisfying all these properties. Now, by property~\eqref{A1}, it is not possible for the total friendliness to increase by $\beta n^{3/2}$ in each of the $K$ iterations. So, by property~\eqref{A2}, there must be some $k$ for which $A_k\cup B_k$ is $\varepsilon$-good, meaning that there are at most $\varepsilon n$ unfriendly vertices in $A_k$ or at most $\varepsilon n$ unfriendly vertices in $B_k$. The third property~\eqref{A3} now ensures that there are at most $2\varepsilon n+o(n)$ unfriendly vertices in total at this stage. The bisection $A_k \cup B_k$ has the properties we desire, proving the result.
\end{proof}

\subsection{Overview of the proofs of the key lemmas}
We now briefly discuss the proofs of \cref{lem:total-friendliness,lem:swap-increment,lem:symmetry}. First, \cref{lem:total-friendliness} is proved via a Chernoff bound and a simple union bound over all possible bisections. Second, \cref{lem:swap-increment} is also proved by a union bound: we show that that no bisection of the graph has many vertices with friendliness very close to zero, so that there is always some reasonably large gain from swapping unfriendly vertices; here, one must also control the (small) amount of additional unfriendliness potentially introduced between pairs of swapped vertices.

The proof of \cref{lem:symmetry} is by far the most technical ingredient in the proof. At a high level, one runs the iterated swap algorithm on a random graph $G\sim \GG(n,1/2)$, at each step revealing only that information about $G$ (namely, degrees into certain parts) which is necessary to determine the outcome of the $\alpha$-swap procedure. So, at every step, we need to study a degree-constrained random graph model; this is accomplished using graph enumeration techniques in the style of McKay--Wormald~\cite{MW90}. One can track the fraction of vertices that live in prescribed parts at prescribed times inductively, showing via the second moment method in our degree-constrained random graph model that the numbers of different types of vertices are concentrated. However, several obstacles arise naturally due to the presence of complicated conditional distributions, and the need for all of the different `well-conditioned' degree-constrained models (based on different revelations) to converge to a single distribution of degrees. The totality of what must be tracked to implement this argument is contained in \cref{prop:empirical-induction}.

In particular, we note that the first part of the proof (\cref{lem:total-friendliness,lem:swap-increment}) and the second part of the proof (\cref{lem:symmetry}) are essentially logically independent, and the analysis here can be extended to a variety of similar algorithms based on degree sequences. One can think of the first part as providing a monovariant to the graph process analysed in the second part, guaranteeing that the graph partition `gets better' over time and converges to a friendly distribution of degrees rather than to an abstract (iterated) optimiser of some associated variational problem.

\section{Swapping decrement}\label{sec:decrement}
In this section we prove \cref{lem:total-friendliness,lem:swap-increment}. To start with, we need some simple facts about centered binomial distributions. The first is a Chernoff bound (see~\cite[Theorem~2.1]{JLR00}, for example) and the second follows from either Stirling's approximation or the Erd\H os--Littlewood--Offord theorem (see~\cite[Corollary~7.4]{TV10}).
\begin{theorem}\label{thm:binomial}
For $N\in\NN$, let $X_{1},\dots,X_{N}$ be independent
Rademacher random variables (satisfying $\Pr(X_{i}=1)=\Pr(X_{i}=-1)=1/2$),
and let $X=X_{1}+\dots+X_{N}$.
\begin{enumerate}
\item For all $t\ge0$, we have $\Pr(|X|\ge t)\le2e^{-t^{2}/(2N)}$.
\item For all $t\ge1$ and all $x\in\RR$, we have $\Pr(|X-x|\le t)\le \sqrt{2}t/\sqrt{N}$. \qed
\end{enumerate}
\end{theorem}

The proof of \cref{lem:total-friendliness} is extremely simple, being a routine application of the union bound.
\begin{proof}[Proof of \cref{lem:total-friendliness}]
There are $\binom{n}{n/2} \le 2^{n}$ bisections in total. For each such bisection $A\cup B$, the random variable $\Delta_{A,B}+n/2$ has a centered binomial distribution to which \cref{thm:binomial} applies (with $N=\binom{n}{2}$). For
sufficiently large $\gamma$, we then have 
\[
\Pr\left(|\Delta_{A,B}|\ge\gamma n^{3/2}\right)\le2\exp\left(-\frac{(\gamma n^{3/2}-n/2)^{2}}{2\binom{n}{2}}\right)=o(2^{-n}),
\]
so the desired result follows from the union bound.
\end{proof}

\cref{lem:swap-increment} is also proved by the union bound, but for this, we will first need to prove some auxiliary lemmas.
\begin{lemma}\label{lem:mostly-discrepancy}
For any sufficiently small fixed $\eta>0$, a random graph $G\sim\GG(n,1/2)$ with high probability has the property that for every bisection $A\cup B$ of $G$, we have $|\Delta_{A,B}(v)|\ge 4^{-1/\eta}\sqrt{n}$ for all but at most $\eta n$ vertices $v\in A$, and for all but at most $\eta n$ vertices $v\in B$.
\end{lemma}

\begin{proof}
For each bisection $A\cup B$, if we condition on an outcome of $G[A]$, then the random variables $\{\Delta_{A,B}(v):v\in A$\} become mutually independent. Conditionally, for each $v\in A$, the random variable $2\Delta_{A,B}(v)+1$ has a centered binomial distribution to which \cref{thm:binomial} applies (with $N=n-1$). Therefore, 
\[\Pr\left(|\Delta_{A,B}(v)|\le4^{-1/\eta}\sqrt{n}\right)\le(\sqrt{2}\cdot 4^{-1/\eta}\sqrt{n})/\sqrt{n-1}\le2\cdot 4^{-1/\eta}\]
for large $n$, from which it follows that the probability that the property in the
statement of the lemma does not hold is at most
\[
2^{n}\binom{n/2}{\eta n}(2\cdot 4^{-1/\eta})^{\eta n}=o(1). \qedhere
\]
\end{proof}

\begin{lemma}\label{lem:small-set-bound}
For any sufficiently small fixed $\alpha>0$, a random graph $G\sim\GG(n,1/2)$ with high probability has the property that for every bisection $A\cup B$ of $G$ and every pair of subsets $A'\subseteq A$ and $B'\subseteq B$ each of size $\alpha n$, we have 
\[|\Delta_{A',B'}|\le\alpha^{4/3}n^{3/2},\]
where we view $A'\cup B'$ as a bisection of the induced subgraph $G[A'\cup B']$.
\end{lemma}
\begin{proof}
Note that the event does not depend on $A,B$, only on $A',B'$. For subsets $A'$ and $B'$ as in the statement of the lemma, the random variable $\Delta_{A',B'}+\alpha n$ has a centered binomial distribution to which \cref{thm:binomial} applies (with $N=\binom{2\alpha n}{2}$). We then have
\[
\Pr\left(|\Delta_{A',B'}|\ge \alpha^{4/3}n^{3/2}\right)\le2\exp\left(-\frac{(\alpha^{4/3}n^{3/2}-\alpha n)^{2}}{2\binom{2\alpha n}{2}}\right)=o\left(\binom{n}{\alpha n}^{-2}\right),
\]
so the desired result follows from a union bound over all choices of $A'$ and $B'$.
\end{proof}
\begin{lemma}\label{lem:swap-no-damage}
For any sufficiently small fixed $\delta>0$, a
random graph $G\sim\GG(n,1/2)$ with high probability has the following property. For
every bisection $A\cup B$, and every pair of subsets $A'\subseteq A,B'\subseteq B$
each of size $\delta n$, if we swap $A'$ and $B'$ to obtain a bisection $A_1\cup B_1$ with $A_1=(A\setminus A')\cup B'$ and $B_1=(B\setminus B')\cup A'$, then
we have 
\[|\Delta_{A_1,B_1}-\Delta_{A,B}|\le \delta^{1/3} n^{3/2}.\]
\end{lemma}
\begin{proof}
For each bisection $A\cup B$ and subsets $A'$ and $B'$ as in
the lemma statement, the random variable $\Delta_{A_1,B_1}-\Delta_{A,B}$
has a centered binomial distribution to which \cref{thm:binomial} applies (with $N=2(n/2-\delta n)\delta n$).
We then have 
\[
\Pr\left(|\Delta_{A_1,B_1}-\Delta_{A,B}|\ge \delta^{1/3} n^{3/2}\right)\le2\exp\left(-\frac{(\delta^{1/3} n^{3/2})^{2}}{4(n/2-\delta n)\delta n}\right)=o\left(2^{-n}\binom{n/2}{\delta n}^{-2}\right),
\]
so the desired result follows once again from the union bound.
\end{proof}

We are now ready to prove \cref{lem:swap-increment}.
\begin{proof}[Proof of \cref{lem:swap-increment}]
Let $\eta<\varepsilon/2$ be small enough for \cref{lem:mostly-discrepancy} to hold. Let $\alpha\in(0,\varepsilon/2)$ be small enough so that \cref{lem:small-set-bound} holds and \cref{lem:swap-no-damage} holds for $\delta=\alpha^4$, and also $\alpha\le4^{-3/\eta}$. Now assume that the properties in \cref{lem:mostly-discrepancy,lem:small-set-bound,lem:swap-no-damage} all hold for $G$ with these parameters, which occurs with high probability.

Now, consider an arbitrary bisection $A\cup B$ where at least $\varepsilon n$ vertices in $A$ are unfriendly and at least $\varepsilon n$ vertices in $B$ are unfriendly. Let $A'$ be the subset of the $\alpha n$ most unfriendly vertices in $A$, and let $B'\subseteq B$ be the subset of the $\alpha n$ most unfriendly vertices in $B$. By assumption, at least $\varepsilon n$ vertices in $A$ are unfriendly, so at least $(\varepsilon-\alpha)n\ge\eta n$ vertices in $A$ are unfriendly but not as unfriendly as the vertices in $A'$. By \cref{lem:mostly-discrepancy} we deduce that for all $v\in A'$ we have $\Delta_{A,B}(v)\le-4^{-1/\eta}\sqrt{n}$. Similarly, for all $v\in B'$ we have $\Delta_{A,B}(v)\le-4^{-1/\eta}\sqrt{n}$.

Next, let $A''=(A\setminus A')\cup B'$ and $B''=(B\setminus B')\cup A'$
be the parts resulting from the first step in an $\alpha$-swap. We know that $|\Delta_{A',B'}|\le\alpha^{4/3}n^{3/2}$
by \cref{lem:small-set-bound}, so we have
\begin{align*}
\Delta_{A'',B''} &= \Delta_{A,B} - 4\sum_{v\in A'\cup B'}\Delta_{A,B}(v) + 4\Delta_{A',B'}
\\
&\ge\Delta_{A,B}+4(2\alpha n)(4^{-1/\eta}\sqrt{n})-4\alpha^{4/3}n^{3/2}\ge\Delta_{A,B}+4\alpha 4^{-1/\eta}n^{3/2}
\end{align*}
Finally, by the guarantee in \cref{lem:swap-no-damage}, we note that the final random swap in the definition of the $\alpha$-swap procedure changes the friendliness of the bisection $A'' \cup B''$ by at most 
\[\delta^{1/3} n^{3/2}=\alpha^{4/3}n^{3/2}\le\alpha4^{-1/\eta}n^{3/2}\] in passing to the final bisection $A_1 \cup B_1$. It follows that we have the desired result with $\beta=3\alpha4^{-1/\eta}$.
\end{proof}

\section{Concentration of the iterated swapping process}\label{sec:concentration}
In this section we prove \cref{lem:symmetry}. In fact, it will follow from the more technical \cref{prop:empirical-induction}, which we shall shortly state and prove by induction. To get started, we need some definitions.

First, we introduce some notation to handle empirical distributions. Given a sequence $(a_i: i \in I)$, the \emph{uniform measure $\mc{\wh{L}}$} on this sequence is the probability distribution of $a_j$ where $j$ is chosen uniformly from $I$. When the sequence $(a_i: i \in I)$ is itself random --- for example, comprised of jointly random vectors --- we emphasise that the associated uniform measure $\mc{\wh{L}}$ is itself a random object, i.e., each realisation of the random sequence $(a_i: i \in I)$ gives rise to an associated uniform measure on this realisation.

We now define some empirical degree distributions associated with our iterated swapping process.
\begin{definition}\label{def:simple-iterated-swap}
Given a graph $G$ on the vertex set $\{1, \dots, n\}$, we consider the iterated swapping process in which we start with the bisection $A_0\cup B_0$, where $A_0=\{1,\dots,n/2\} $ and $B_0=\{n/2+1,\dots,n\} $, and repeatedly perform $\alpha$-swaps $k$ times to yield a sequence $(A_t \cup B_t)_{t= 0}^k$ of bisections. For a binary sequence $x = (x_t)_{t=1}^{k+1} \in\{ 0,1\}^{k+1}$, let $V_x$ be the set of vertices that are in part $A_t$ at those times $t$ with $x_{t-1}=0$, and in part $B_t$ at those times $t$ with $x_{t-1}=1$ for $1\le t\le k+1$. For a binary sequence $x\in\{0,1\}^{k+1}$, let $\mc{\wh{L}}_x$ be the uniform measure on the sequence of degree vectors
\[
\left(\left((\deg_{V_y}(v)-|V_y|/2)/\sqrt{n}\right)_{y\in\{0,1\}^{k+1}}:v\in V_x\right).
\]
\end{definition}

Next, we recall the definition of multidimensional Kolmogorov distance on $\mb{R}^d$.
\begin{definition}
Let $\mathcal L$ and $\mathcal L'$ be probability distributions on $\mb{R}^d$.
We define the \emph{Kolmogorov distance} $\operatorname d_{\mathrm{K}}(\mathcal L,\mathcal L')$ between $\mathcal L$ and $\mathcal L'$ to be the supremum of
$|\mathcal L(A)-\mathcal L'(A)|$ over all sets $A$ of the form $(-\infty,a_1]\times \dots\times (-\infty,a_d]$, where $a_1, \dots, a_d \in \mb{R}$.
\end{definition}

Note that the Kolmogorov distance controls the probability of lying in any half-open box: indeed, for any such box $B=(b_1,c_1]\times\dots\times(b_d,c_d]$, we can use the inclusion-exclusion principle to express $\mathcal L(B)$ as a signed sum of $2^d$ probabilities of the form $\mathcal L((-\infty,a_1]\times\dots\times(-\infty,a_d])$, so $|\mathcal L(B)-\mathcal L'(B)|\le 2^d\operatorname{d}_{\mathrm K}(\mc L,\mc L')$.

The promised generalisation of \cref{lem:symmetry} is now as follows.

\begin{proposition}\label{prop:empirical-induction}
Fix $\alpha\in(0,1/4)$ and $k\in\NN$. There are $c_{\alpha,k}, C_{\alpha,k} > 0$ such that for each $x\in\{0,1\}^{k+1}$ there are
\begin{enumerate}
    \item a $2^{k+1}$-dimensional probability distribution $\mathcal{L}_x$, and
    \item a real number $\pi_x\ge\alpha^{4k}/2$,
\end{enumerate}  
both of which may depend on $\alpha$ and $n$, such that the following holds. For $G\sim\GG(n,1/2)$, consider a sequence of $k$ iterated $\alpha$-swaps, and for $x\in\{0,1\}^{k+1}$, let $V_x$ and $\mc{\wh{L}}_x$ be as in \cref{def:simple-iterated-swap}. Then, with high probability, all of the following hold.
\begin{enumerate}[{\bfseries{A\arabic{enumi}}}]
\item\label{B1} For each $x\in\{0,1\}^{k+1}$, we have 
\[\left|\vphantom{\deg_{V_x}}|V_x|-\pi_x n\right|\le n^{1-c_{\alpha,k}}.\]
\item\label{B2} For each $x\in\{0,1\}^{k+1}$, we have \[\operatorname d_{\mathrm{K}}(\mc{\wh{L}}_x,{\mathcal{L}}_x) \le n^{-c_{\alpha,k}}.\]
\item\label{B3} For each vertex $v \in V(G)$ and each $x\in\{0,1\}^{k+1}$, we
have 
\[\left|{\deg_{V_x}(v)}-|V_x|/2\right|< C_{\alpha,k}\sqrt{n \log n}.\]
\item\label{B4} For each $x\in\{0,1\}^{k+1}$, and each box $B=\prod_{y\in\{0,1\}^{k+1}}(a_y,b_y]$ with side lengths $b_y-a_y=n^{-c_{\alpha,k}}$ (and, therefore, $\vol(B)=(n^{-c_{\alpha,k}})^{2^{k+1}}$) we have 
\[\mc{L}_x(B)\le \vol(B) \exp(C_{\alpha,k}\sqrt{\log n}).\]
\end{enumerate}
\end{proposition}
Again, we emphasise that we treat $\alpha$ and $k$ as fixed constants for the purpose of the `with high probability' statement in the above proposition; in particular, \cref{prop:empirical-induction} only holds if $n$ grows sufficiently fast (with respect to $\alpha$ and $k$).

Before discussing the proof of \cref{prop:empirical-induction}, we explain how it implies \cref{lem:symmetry}. The key observation is that \cref{B1,B2,B3,B4} essentially allow us to read off, from the distributions $\mc L_x$, arbitrary information about degree statistics (and, in particular, the number of friendly vertices in each part). We will need the following lemma.

\begin{lemma}\label{lem:halfspace-simple}
Suppose that $G$ is such that \cref{B2,B3,B4} are satisfied, and let $H\subseteq \RR^{\{0,1\}^{k+1}}$ be any closed half-space (i.e., a region bounded by a hyperplane). Then for any $x\in \{0,1\}^{k+1}$, we have $\wh{\mc L}_x(H)=\mc L_x(H)+o(1)$.
\end{lemma}

We defer the proof of \cref{lem:halfspace-simple} (in a slightly stronger form, see \cref{lem:halfspace}) to \cref{sub:preliminaries-inductive-step}; we now deduce \cref{lem:symmetry} from \cref{prop:empirical-induction,lem:halfspace-simple}.
\begin{proof}[Proof of \cref{lem:symmetry}]
Let $A_k\cup B_k$ be the bisection resulting from $k$ iterations of the $\alpha$-swap process. Recall that in the statement of \cref{lem:symmetry}, the random variables $X$ and $Y$ are the numbers of unfriendly vertices in $A_k$ and $B_k$. It suffices to prove that there is some value $N$ (potentially depending on all of $\alpha,k,n$) such that $X=N+o(n)$ with high probability. Indeed, by symmetry it would follow that $Y=N+o(n)$ with high probability as well, implying that $|X-Y|=o(n)$ with high probability, as desired.

To this end, for $i\in \{0,1\}$, let $S_i=\{x\in \{0,1\}^{k+1}:x_{k+1}=i\}$ and note that a vertex $v\in A_k$ is unfriendly if and only if
\[\sum_{y\in S_0}\deg_{V_y}(v)-\sum_{y\in S_1}\deg_{V_y}(v)=\sum_{y\in S_0}\left(\deg_{V_y}(v)-|V_y|/2\right)-\sum_{y\in S_1}\left(\deg_{V_y}(v)-|V_y|/2\right)\le 0.\]
So, defining the affine half-space
\[H=\left\{d\in \RR^{\{0,1\}^{k+1}}:\sum_{y\in S_0}d_y-\sum_{y\in S_1}d_y\le 0\right\},\]
we have $X=\sum_{x\in S_0}|V_x|\wh{\mc L}_x(H)$.
By \cref{prop:empirical-induction} and \cref{lem:halfspace-simple}, with high probability we have $X=n\sum_{x\in S_0}\pi_x \mc L_x(H)+o(n)$, as desired.
\end{proof}

We will prove \cref{prop:empirical-induction} by induction on $k$. In its full generality, our argument will rely on a second moment computation that utilises results of McKay--Wormald~\cite{MW90} and Canfield--Greenhill--McKay~\cite{CGM08} about enumerating graphs with specified vertex-degrees. Since the argument is rather technical, we shall proceed slowly, first illustrating the base case before jumping into the meat of the argument.

\subsection{The base case}\label{sub:base-case}
In this subsection we prove \cref{prop:empirical-induction} for $k = 0$. This entails some explicit calculations in the random graph $\mb{G}(n,1/2)$; the inductive step can be seen as a `relativised' version of this argument, with the randomness coming from a well-conditioned random graph with specified degree information rather than $\mb{G}(n,1/2)$.

Recall that we need to prove that the four properties in \cref{B1,B2,B3,B4} each hold with high probability. The most interesting of these properties is \cref{B2}, which will be established using the following lemma.

\begin{lemma}\label{lem:empirical-boxes}
Fix $c>0$ and $d\in \NN$. Let $(\vec d(v))_{v\in V}$ be a sequence of $n$ discrete jointly random vectors in $\RR^d$, and let $\mc L$ be the (fixed) distribution on $\mb{R}^d$ defined by choosing $v$ uniformly at random from $V$ and then sampling from $\vec d(v)$. Suppose that for a box $Q=(-q,q]^d$ with $q\ge 1$, the following conditions hold:
\begin{enumerate}
    \item\label{leb1} for each $\vec s,\vec t\in Q$ and each $u,v\in V$, we have
    \[\Pr(\vec d(u)=\vec t \text{ and } \vec d(v)=\vec s)=(1\pm n^{-c})\Pr(\vec d(u)=\vec t)\Pr(\vec d(v)=\vec s),\]    
    \item\label{leb2} $\mc L(Q^c)\le n^{-c}$, and
    \item\label{leb3} for each box $B\subseteq Q$ with side lengths at least $n^{-c}$, we have $\mc L(B)\le q \vol(B)$.
\end{enumerate}
For a given realisation of the random sequence $(\vec d(v))_{v\in V}$, let $\wh{\mc L}$ be the (random) distribution on $\mb{R}^d$ which is the uniform measure on this realisation. With probability at least $1-O(q^dn^{-c/8})$ over the randomness of $(\vec d(v))_{v\in V}$, we have $\operatorname{d}_{\mathrm K}(\mc L,\wh{\mc L}) = O(q^dn^{-c/(8d)})$.
\end{lemma}

In applications, $\vec d(v)$ will be a list of degrees from $v$ to a number of other fixed subsets, and $(\vec d(v))_{v\in V}$ will be the random ensemble of these lists. The above lemma roughly states that given decorrelation between these degree statistics, and (for technical reasons) a tail bound and anti-concentration, the empirical degree distribution of $V$ is very likely to concentrate around an explicit distribution.

Here, we again reiterate that the constants suppressed by the asymptotic notation in \cref{lem:empirical-boxes} are allowed to depend on the fixed parameters $c$ and $d$.
\begin{proof}[Proof of \cref{lem:empirical-boxes}]
For any $v\in V$, and any box $B$, let $\mc E_{v,B}$ be the event that $\vec d(v)$ lies in $B$, so that $n\mc{\wh{L}}(B)$ is the number of $v\in V$ such that $\mc E_{v,B}$ holds. For $u,v\in V$ and boxes $B,B'\subseteq Q$, we can sum the bound in~\eqref{leb1} over all the points $\vec t\in B$ and $\vec s\in B'$ to see that 
\[\Pr(\mc E_{u,B}\cap \mc E_{v,B'})=\Pr(\mc E_{u,B})\Pr(\mc E_{v,B'})\pm n^{-c}.\] 
It follows that $\Var(n \mc{\wh{L}}(B))\le n+n^{2-c}\le 2n^{2-c}$, so by Chebyshev's inequality, with probability at least $1-n^{-c/4}$, we have
\begin{equation}
    \left|\wh{\mc L}(B)-\E[\wh{\mc L}(B)]\right|=\left|\wh{\mc L}(B)-{\mc L}(B)\right|\le 2n^{-c/4}.\label{eq:discretised-estimate}
\end{equation}
Now, consider a family $\mf{B}$ of $O(n^{c/8}q^d)$ half-open boxes with side lengths at most $D=n^{-c/(8d)}$ that partition the (big) box $Q$. By the union bound, with probability $1-O(q^dn^{-c/8})$, the bound \cref{eq:discretised-estimate} holds for all $B\in \mf{B}$. Also, since $\E [\wh{\mc L}(Q^c)]={\mc L}(Q^c)\le n^{-c}$, by Markov's inequality we have $\wh{\mc L}(Q^c)\le n^{-c/2}$ with probability at least $1-n^{-c/2}$. Now, it is a routine matter to deduce the desired conclusion from these two facts. The details are as follows.

For any semi-infinite box $A=(-\infty,a_1]\times \dots\times (-\infty,a_d]$, we can find subcollections $\mf{B}_-,\mf{B}_+\subseteq \mf{B}$ such that \[\bigcup_{B\in \mf{B}_-}B\subseteq A\cap Q\subseteq \bigcup_{B\in \mf{B}_+}B,\]
and $|\mf{B}_+\setminus \mf{B}_-|=O((q/D)^{d-1})$. Then
\[
\sum_{B\in \mf{B}_-}\mc{\wh{L}}(B)\le \mc{\wh{L}}(A\cap Q)\le \sum_{B\in \mf{B}_+}\mc{\wh{L}}(B).
\]
Furthermore, using~\eqref{leb3} and \cref{eq:discretised-estimate} for all $B\in \mf B$, we see that both the sum $\sum_{B\in \mf{B}_-}\mc{\wh{L}}(B)$ and the sum $\sum_{B\in \mf{B}_+}\mc{\wh{L}}(B)$ differ from $\mathcal L(A\cap Q)$ by at most
\[O\left(|\mf{B}_+\setminus \mf{B}_-|(qD^d)+|\mf{B}|(2n^{-c/4})\right)=O\left(q^dn^{-c/(8d)}\right).\]
So, we have
\[\left|\mathcal L(A)-\mc{\wh{L}}(A)\right|=O\left(\mc L(Q^c)+\wh{\mc L}(Q^c)+q^dn^{-c/(8d)}\right)=O\left(q^dn^{-c/(8d)}\right),\]
proving the lemma.
\end{proof}

Now we use \cref{lem:empirical-boxes} to prove the base case of \cref{prop:empirical-induction}.

\begin{proof}[Proof of the $k=0$ case of \cref{prop:empirical-induction}]
First, we have $|V_0| = |A_0| = |V_1| = |B_0| = n/2$, proving \cref{B1}. Furthermore, for a sufficiently large $C_{\alpha,k}>0$, given a vertex $v$, we have $|{\deg_{V_i} (v)}-n/4|< C_{\alpha,k}n^{1/2}\sqrt{\log n}$ with probability at least $1-1/n^2$, say, just by the Chernoff bound, whence a union bound demonstrates \cref{B3}.

It remains to prove \cref{B2,B4}. It is enough to prove them for $x=(0)$, by symmetry. We will take $\mathcal L_0$ to be the distribution of the random vector 
\[\vec d(v) = \left(|{\deg_{V_0}(v)}-n/4|/\sqrt n,\;|{\deg_{V_1}(v)}-n/4|/\sqrt n\right),\] where $v\in V_0$ is arbitrary; clearly, this distribution does not actually depend on the specific choice of $v \in V_0$. Then, $\mathcal L_0$ has a simple description in terms of independent binomial distributions. Although it will not be necessary for the proof, we remark that $\mathcal L_0$ is well-approximated by the bivariate normal distribution $N(0,1/2)^2$, and it is possible to take $\mathcal L_0$ to be this distribution as well.

Before proceeding further, we note that the aforementioned Chernoff bound shows that with $Q=(-C_{\alpha,k}\sqrt {\log n},C_{\alpha,k}\sqrt {\log n}]^2$, we have $\mc L_0(Q^c)\le 2/n^2$. Now, for every individual point $\vec d\in \RR^2$, we have $\mathcal L_0(\{{\vec d}\})=O((1/\sqrt n)^2)=O(1/n)$ (by the Erd\H os--Littlewood--Offord theorem applied to each coordinate, say). Since $\mc L_0$ is supported on the lattice $((\mb Z-n/4)/\sqrt n)^2$, for a box $B$ with side lengths at least $1/\sqrt n$, we have
\begin{equation}
    \mathcal L_0(B) = O\left(\vol(B)\right),\label{eq:gaussian-anticoncentration}
\end{equation}
establishing \cref{B4}. Now, we claim that for every pair of vertices $u,v$ and every pair of points $\vec s,\vec t\in Q$, we have
\[\Pr(\vec d(u)=\vec t\text{ and }\vec d(v)=\vec s)=(1 \pm O(\sqrt{\log n/n}))\Pr(\vec d(u)=\vec t)\Pr(\vec d(v)=\vec s).\]
Indeed, we will then be able to apply \cref{lem:empirical-boxes} to establish that \cref{B2} holds with high probability. The claim follows from the following explicit calculation. The only dependence between $\vec d(u)$ and $\vec d(v)$ comes from the potential edge between $u$ and $v$, but we can check that if we condition on this edge being present (or not), the probabilities $\Pr(\vec d(u)=\vec t)$ and $\Pr(\vec d(v)=\vec s)$ vary only by a factor of $(1\pm O(\sqrt{\log n/n}))$, which in itself boils down to the observation that $\binom{n/2-1}{t}/\binom{n/2-1}{t-1}=(n/2-t)/t=1+O(1/4-t/n)$.
\end{proof}

\subsection{Preliminaries for the inductive step}\label{sub:preliminaries-inductive-step}
We start with some preparations before proceeding to the details of the inductive step. First, we provide a proof of \cref{lem:halfspace-simple}; actually we prove the following more general lemma.
\begin{lemma}\label{lem:halfspace}
For fixed $c>0,d\in \NN$ and any $q\ge 1$, let $\mc L,{\mc L}'$ be probability distributions on $\RR^d$ satisfying $\operatorname d_{\mathrm K}(\mc L,{\mc L}')\le n^{-c}$, ${\mc L}'\left((-q,q]^d\right)=1$, and $\mc L(B)\le q\vol(B)$ for all boxes $B$ with side lengths at least $n^{-c}$. Then the following conclusions hold.
\begin{enumerate}
    \item For any region $H\subseteq \RR^d$ defined as the intersection of $O(1)$ (closed or open) affine half-spaces, we have ${\mc L}'(H) = \mc L(H) \pm  O(q^dn^{-c/(2d)})$.
    \item For any $R\subseteq \RR^d$ obtained as the region between two parallel (closed or open) affine hyperplanes separated by a distance of at most $n^{-c}$, we have $\mc L(R)= O(q^dn^{-c/(2d)})$.
\end{enumerate}
\end{lemma}
Here, the constants suppressed by the asymptotic notation in \cref{lem:halfspace} are allowed to depend on the fixed parameters $c$ and $d$.

\begin{proof}[Proof of \cref{lem:halfspace}]
Let $Q=(-q,q]^d$, and note that $\mc L(Q)={\mc L}'(Q) \pm O(n^{-c})=1 - O(n^{-c})$. As in the proof of the base case of \cref{prop:empirical-induction} (in \cref{sub:base-case}), we consider a family $\mf{B}$ of $O(q^dn^{c/2})$ half-open boxes with side lengths at most $D=n^{-c/(2d)}$ that partition $Q$.

For the first point, let $\mf B_+\subseteq \mf B$ be the subcollection of boxes which intersect $H$, and let $\mf B_-\subseteq \mf B$ be the subcollection of boxes fully included in $H$, so that $|\mf{B}_+\setminus \mf{B}_-|=O((q/D)^{d-1})$. We then observe that $|{\mc L}'(H)-{\mc L}(H)|$ is bounded by
\[O\left(|\mf{B}_+\setminus \mf{B}_-|D^d q+|\mf{B}|n^{-c}+\mc L(Q^c)\right)=O\left(q^dn^{-c/(2d)}\right).\]

For the second part, let $\mf B_+$ be the subcollection of boxes that intersect $R$, so $|\mf{B}_+|=O((q/D)^{d-1})$. We similarly observe that \[\mc L(R)=O\left(|\mf{B}_+|(D^dq+n^{-c})+\mc L(Q^c)\right)=O\left(q^dn^{-c/(2d)}\right). \qedhere\]
\end{proof}

Second, we isolate the part of the proof of \cref{lem:empirical-boxes} in which we approximated Kolmogorov distance via small boxes.

\begin{lemma}\label{lem:box-kolmogorov}
For fixed $c>0$ and $d\in\NN$, there exists a $c' = c'(c,d) >0$ for which the following holds. Let $\mc L,\mc L'$ be probability distributions on $\RR^d$, where $\mc L'$ is (possibly) a random object. Let $Q=(-q,q]^d\subseteq \RR^d$ be a box for $q\ge 1$, and let $\mf B$ be a partition of it into at most $q^d n^{c/2}$ boxes with side lengths at most $n^{-c/(2d)}$. Suppose the following conditions are satisfied.
\begin{enumerate}
    \item For each $B\in \mf B$, we have $|\mc L'(B)-\mc L(B)|\le n^{-c}$ with probability at least $1-n^{-c}$.
    \item $\mc L(Q^c) \le n^{-c}$, and $\mc L'(Q^c)\le n^{-c}$ with probability at least $1 - n^{-c}$.
    \item For each box $B\in \mf B$ with side lengths at least $n^{-c}$, we have $\mc L(B)\le q \vol(B)$.
\end{enumerate}
Then, with high probability, we have $\operatorname{d}_{\mathrm K}(\mc L,\mc L')\le q^dn^{-c'}$. \qed
\end{lemma}

We will also need some lemmas for working with random graphs with constrained degree sequences. These lemmas will be deduced from powerful enumeration theorems due to McKay and Wormald~\cite{MW90} and Canfield, Greenhill, and McKay~\cite{CGM08}. Before stating these lemmas, we define a notion of `closeness' between two degree sequences. This definition is chosen to be convenient for the proof of \cref{prop:empirical-induction}; it has two cases which will both arise in different parts of the proof.

\begin{definition}\label{def:close}
Consider a pair of sequences $(a(v))_{v\in V}$ and $(b(w))_{w\in W}$. Let $\wh{\mc A},\wh{\mc B}$ be the uniform measures on these sequences (obtained by choosing a random element of each of these sequences). We say that $(a(v))_{v\in V}$ and $(b(w))_{w\in W}$ are \emph{proximate} if at least one of the following two conditions holds.
\begin{enumerate}
    \item There is a bijection $\psi:V\to W$ such that $\sum_{v\in V} |a(v)-b(\psi(v))| = O(|V|)$.
    \item $\left||V|-|W|\right|\le n^{1-\Omega(1)}$ and $\operatorname{d}_{\mathrm K}(\wh{\mc A},\wh{\mc B})\le n^{-\Omega(1)}$.
\end{enumerate}
\end{definition}

We are now ready to state the promised pair of lemmas. We defer the details of their proofs to \cref{app:MW}. The first of these lemmas is for the non-bipartite setting. Recall that $\simeq$ means equality up to a multiplicative factor $(1\pm n^{-\Omega(1)})$.

\begin{lemma}\label{lem:enum-consequence}
Let $(d_w)_{w\in W}$ be a sequence with even sum on a set $W$ of $n$ vertices such that 
\begin{itemize}
\item $d_w=n/2 \pm O(\sqrt{n\log n})$ for each $w \in W$, 
\item $\sum_{w\in T} d_w=n|T|/2 \pm O(n^{3/2})$ for all $T \subseteq W$, and 
\item $\sum_{w \in W} (d_w - n/2)^2 = O(n^2)$.
\end{itemize} 
Such a sequence is a graphic sequence for all sufficiently large $n$. Let $G$ be a uniformly random graph on $W$ with this degree sequence. Then, for any fixed $v \in W$ and $S\subseteq W$ satisfying $|S|,n-|S|=\Omega(n)$, the following hold.
\begin{enumerate}
    \item\label{enumcons1} For any integer $0 \le t \le |S|$, parameterising $t=|S|/2+\tau\sqrt n$, if $|\tau|  > n^{1/10}$, then we have
    \[\Pr(\deg_S(v)=t) \le \exp(-\Omega(\tau^2)),\]
    and if $|\tau| \le n^{1/10}$, then we have
    \[\Pr(\deg_S(v)=t) \le \exp\left(O\left(|\tau|+\sqrt{\log n}\right)\right) \Pr(Z=t),\]
    where $Z=|R\cap S|$ for a random subset $R\subseteq W$ of size $d_v$, i.e., 
    \[Z\sim \operatorname{Hypergeometric}(n,|S|, d_v).\]
    \item\label{enumcons2}Let us write 
    \[\Pr(\deg_S(v)=t)=p(v,(d_w)_{w\in S},(d_w)_{w\notin S},t)\] 
    as a function of $v$, the relevant degree sequences, and $t$. Then, for $t=|S|/2\pm O(\sqrt{n \log n})$ and the other parameters as constrained above, this function $p(\cdot)$ depends continuously on its parameters, in the following sense: if
    \begin{itemize}
        \item $|t-t'|,|d_v-d_{v'}'|\le n^{1/2-\Omega(1)}$,
        \item $(d_w)_{w\in S}$ and $(d_w')_{w\in S'}$ are proximate, and
        \item $(d_w)_{w\in W\setminus S}$ and $(d_w')_{w\in W'\setminus S'}$ are proximate,
    \end{itemize}
    then 
    \[p(v,(d_w)_{w\in S},(d_w)_{w\in W\setminus S},t)\simeq p(v',(d'_w)_{w\in S'},(d'_w)_{w\in W'\setminus S'},t'),\]
    recalling that $\simeq$ denotes equality up to a multiplicative factor of $1 \pm n^{-\Omega(1)}$.
\end{enumerate}
\end{lemma}

Next, the second of the promised pair of lemmas is for the bipartite setting. 
\begin{lemma}\label{lem:enum-consequence-bipartite}
Let $((d_v)_{v\in V}, (d_w)_{w\in W})$ be a pair of sequences with identical sums on a bipartition $V\cup W$ with $|V|,|W|=\Theta(n)$ such that
\begin{itemize}
\item $d_v=|W|/2 \pm O(\sqrt{n\log n})$ for all $v\in V$ and $d_w=|V|/2 \pm O(\sqrt{n\log n})$ for all $w\in W$, 
\item $\sum_{v\in T} d_v=|W||T|/2 \pm O(n^{3/2})$ for all $T \subseteq V$ and $\sum_{w\in T} d_w=|V||T|/2 \pm O(n^{3/2})$ for all $T \subseteq W$, and
\item $\sum_{v \in V} (d_v - |W|/2)^2 = O(n^2)$ and $\sum_{w \in W} (d_w - |V|/2)^2 = O(n^2)$.
\end{itemize}
Such a pair of sequences form a bipartite-graphic sequence for all sufficiently large $n$. Let $G$ be a uniformly random bipartite graph between $V$ and $W$ with this degree sequence. Then, for any fixed $u \in V$ and $S\subseteq W$ satisfying $|S|,n-|S|=\Omega(n)$, the following hold.
\begin{enumerate}
    \item\label{enumconsbip1} For any integer $0 \le t \le |S|$, parameterising $t=|S|/2+\tau\sqrt n$, if $|\tau|  > n^{1/10}$, then we have
    \[\Pr(\deg_S(u)=t) \le \exp(-\Omega(\tau^2)),\]
    and if $|\tau| \le n^{1/10}$, then we have
    \[\Pr(\deg_S(u)=t) \le \exp\left(O\left(|\tau|+\sqrt{\log n}\right)\right) \Pr(Z=t),\]
    where $Z=|R\cap S|$ for a random subset $R\subseteq W$ of size $d_v$, i.e., 
    \[Z\sim \operatorname{Hypergeometric}(|W|,|S|, d_v).\]
    \item\label{enumconsbip2} Let us write 
    \[\Pr(\deg_S(u)=t)=p(u,(d_v)_{v\in V}, (d_w)_{w\in S},(d_w)_{w\in W\setminus S},t)\] 
    as a function of $u$, the relevant degree sequences, and $t$. Then, for $t=|S|/2\pm O(\sqrt{n \log n})$ and the other parameters as constrained above, this function $p(\cdot)$ depends continuously on its parameters, in the following sense: if
    \begin{itemize}
        \item $|t-t'|,|d_u-d'_{u'}|\le n^{1/2-\Omega(1)}$,
        \item $(d_v)_{v\in V}$ and $(d'_v)_{v\in V'}$ are proximate,
        \item $(d_w)_{w\in S}$ and $(d_w')_{w\in S'}$ are proximate, and
        \item $(d_w)_{w\in W\setminus S}$ and $(d'_w)_{w\in W'\setminus S'}$ are proximate,
    \end{itemize}
    then 
    \[p(u, (d_v)_{v\in V}, (d_w)_{w\in S},(d_w)_{w\in W\setminus S},t)\simeq p(u', (d'_v)_{v\in V}, (d'_w)_{w\in S'},(d'_w)_{w\in W'\setminus S'},t'),\]
    recalling that $\simeq$ denotes equality up to a multiplicative factor of $1 \pm n^{-\Omega(1)}$.
\end{enumerate}
\end{lemma}

Finally, we require the following concentration properties of the edge-counts in a random graph. 
\begin{lemma}\label{lem:total-disc-concentration}
There are absolute constants $C, c > 0$ such that if $G\sim\GG(n,1/2)$ is a random graph, then with probability at least $1-2\exp(-cn)$ we have for all disjoint $S,T$ that
\begin{enumerate}
\item $\sum_{v\in T}(\deg_S(v)-|S|/2)^2\le Cn^2$,
\item $\sum_{v\in T}(\deg_T(v)-(|T|-1)/2)^2\le Cn^2$,
\item $|\sum_{v\in T}(\deg_S(v)-|S|/2)|\le Cn^{3/2}$, and
\item $|\sum_{v\in T}(\deg_T(v)-(|T|-1)/2)|\le Cn^{3/2}$. \qed
\end{enumerate}
\end{lemma}
The proof of \cref{lem:total-disc-concentration} is an immediate application of a Chernoff bound and the union bound, similar to the proof of  \cref{lem:total-friendliness}, so we omit the details.

Now we are ready to finish the proof of \cref{prop:empirical-induction} by establishing its inductive step.
\subsection{Proof of the inductive step}\label{sub:inductive-step}

Consider $k-1$ iterations of the $\alpha$-swap process, giving rise to a partition of the vertices into sets $V_x$, for $x\in \{0,1\}^k$, as defined in \cref{def:simple-iterated-swap}. An additional iteration of the $\alpha$-swap process will refine this to a partition into sets $V_x$, for $x\in \{0,1\}^{k+1}$; to emphasise the difference between these two partitions we write $W_x$ instead of $V_x$ when $x\in \{0,1\}^k$.

By the inductive hypothesis, there are real numbers $\pi_x\ge \alpha^{4(k-1)}/2$ and distributions $\mc{L}_x$ for $x \in \{0,1\}^k$ such that the following properties are satisfied with high probability.
\begin{enumerate}[{\bfseries{B\arabic{enumi}}}]
\item\label{C1} For each $x\in\{0,1\}^{k}$, we have 
\[\left|\vphantom{\deg_{V_x}}|W_x|-\pi_x n\right|\le n^{1-c_{\alpha,k-1}}.\]
\item\label{C2} For each $x\in\{0,1\}^{k}$, we have 
\[\operatorname d_{\mathrm{K}}(\mc{\wh{L}}_x,{\mathcal{L}}_x) \le n^{-c_{\alpha,k-1}}.\]
\item\label{C3} For each vertex $v \in V(G)$ and each $x\in\{0,1\}^{k}$, we
have 
\[\left|{\deg_{W_x}(v)}-|W_x|/2\right|\le C_{\alpha,k-1}n^{1/2}\sqrt{\log n}.\]
\item\label{C4} For each $x\in\{0,1\}^{k}$, and each box $B$ with side lengths $n^{-c_{\alpha,k-1}}$ we have 
\[\mc{L}_x(B)\le \vol(B) \exp(C_{\alpha,k-1}\sqrt{\log n}).\]
\end{enumerate}

Here, we remind the reader that $\mc{\wh{L}}_x$ is an empirical distribution measuring the degrees of vertices in $W_x$ into the various sets $W_y$. Also, we remark that although \cref{C4} as written only concerns boxes with side lengths \emph{exactly} $n^{-c_{\alpha,k}}$, a simple covering argument shows that the same conclusion holds when $B$ is a box with side lengths \emph{at least} $n^{-c_{\alpha,k}}$ (up to a constant factor).

Next, let
\[\mc{R}=\left((W_x)_{x\in \{0,1\}^k},(\deg_{W_x}(v))_{v\in V(G),x\in \{0,1\}^k}\right)\]
record the part and degree information after $k-1$ iterations of the $\alpha$-swap process, so \cref{C1,C2,C3,C4} are all really properties of $\mc R$. Let $\mathcal E$ be the event that all the conclusions of \cref{lem:total-disc-concentration} hold for all disjoint subsets of vertices $S$ and $T$. By \cref{lem:total-disc-concentration}, we have \[\Pr({\mc E}^c)=\E\left[\Pr({\mc E}^c\, | \,\mc R)\right]\le e^{-cn}\] for some universal $c>0$, so by Markov's inequality, with high probability, $\mc R$ has the property that
\begin{enumerate}[{\bfseries{B}\arabic{enumi}}, resume]
\item\label{C5} $\Pr(\mc E \, | \, \mc R)\ge 1- e^{-(c/2)n}$.
\end{enumerate}

Now, let us condition on an outcome of $\mc R$ satisfying \cref{C1,C2,C3,C4,C5}; we say that such an outcome is \emph{well-behaved}. It suffices to prove that, in the resulting conditional probability space, \cref{B1,B2,B3,B4} hold with high probability. Note that, conditionally, $G$ is now a random graph with certain degree constraints. To be precise, for each $x\in \{0,1\}^k$, the induced subgraph $G[W_x]$ is uniform over all graphs in which each $v\in W_x$ has degree $\deg_{W_x}(v)$, and for each pair of distinct $x,y\in \{0,1\}^k$, the subgraph $G[W_x,W_y]$ (consisting of the edges of $G$ between $W_x$ and $W_y$) is uniform over all bipartite graphs in which each $v\in W_x$ has degree $\deg_{W_y}(v)$ and each $v\in W_y$ has degree $\deg_{W_x}(v)$. Furthermore, all these random subgraphs of the form $G[W_x],G[W_x,W_y]$ are independent, and \cref{C1,C3,C5} in particular ensure that either \cref{lem:enum-consequence} or \cref{lem:enum-consequence-bipartite} apply to all these subgraphs.

Recalling that we have performed $k-1$ iterations of the $\alpha$-swap procedure so far, we now consider the effect of a $k$th $\alpha$-swap. Recall that this $\alpha$-swap has two steps. First, the $\lfloor\alpha n\rfloor$ unfriendliest vertices on each side are swapped. The information recorded in $\mc R$ is enough to determine the outcome of this first step.  Second, a random set of $\lfloor\alpha^4n\rfloor$ vertices on each side are swapped; let $\mathcal S$ be the random pair of sets that are swapped in this second step, and note that $\mathcal S$ is independent from $G$ conditional on the partition at that time.

For the remainder of this proof, asymptotic notation should be understood to be treating $k,\alpha$ as fixed constants, so, for example, the inequality in \cref{C2} can be described as saying $\operatorname d_{\mathrm{K}}(\mc{\wh{L}}_x,{\mathcal{L}}_x) \le n^{-\Omega(1)}$.

\subsubsection{Concentration of the part sizes}\label{subsec:concentration}
First we prove that \cref{B1} holds with high probability. Let $S_i=\{z\in \{0,1\}^{k}:z_{k}=i\}$, and recall that the bisection resulting from the first $k-1$ iterations of the $\alpha$-swap process has parts $A_{k-1}=\bigcup_{z\in S_0}W_z$ and $B_{k-1}=\bigcup_{z\in S_1}W_z$. (Recall that $z_k$ records whether a vertex is in $A_{k-1}$ or $B_{k-1}$.)

Consider any $z\in\{0,1\}^k$, and let $W_{z}'$ be the portion of $W_z$ that is swapped during the first step of the $k$th $\alpha$-swap (i.e., these vertices are among the $\lfloor \alpha n\rfloor$ unfriendliest vertices in their part of the bisection $A_{k-1}\cup B_{k-1}$; this is determined by the outcome of $\mc R$ we have conditioned on). It suffices to prove that $|W_{z}'| = \pi_{z}' n \pm n^{1-\Omega(1)}$, for some $\pi_{z}'$ that does not depend on the specific choice of $\mc{R}$ that we are conditioning on (but demanding no lower bound on $\pi_{z}'$). Indeed, for any $b\in \{0,1\}$, the second part of the $\alpha$-swap process (in which we randomly swap sets $A',B'$ of $\lfloor \alpha^4 n\rfloor$ vertices on both sides) will then, with high probability, yield $|V_{(z,b)}| = \pi_{(z,b)}n \pm n^{1-\Omega(1)}$, where 
\begin{align*}\pi_{(z,b)}&=\begin{cases}
\alpha^4\pi_{z}'+ (1-\alpha^4)(\pi_z-\pi_{z}')&\text{if $z_k=b$}\\
\alpha^4(\pi_z-\pi_{z}')+(1-\alpha^4)\pi_{z}'&\text{if $z_k\ne b$}
\end{cases}\\
&\ge \alpha^4\pi_z\ge \alpha^4\cdot\alpha^{4(k-1)}/2 =\alpha^{4k}/2.
\end{align*}
Here we have used \cref{C1} and a Chernoff bound for the hypergeometric distribution; see for example~\cite[Theorem~2.10]{JLR00}.

To this end, we study the sets $W_{z}'$. Assume without loss of generality that $z_k=0$ (i.e., $W_{z}'\subseteq A_{k-1}$). Let $A'$ be the set of the $\lfloor\alpha n\rfloor$ unfriendliest vertices in $A_{k-1}$ (so $W_{z}'=W_z\cap A'$), and let $A^{(\zeta)}$ be the set of vertices in $A_{k-1}$ with friendliness at most $\zeta\sqrt n$. We will approximate $A'$ with $A^{(\zeta)}$, for an appropriate choice of $\zeta$.

For $\zeta\in \RR$, define the affine half-space
\[H_\zeta=\left\{d\in \RR^{\{0,1\}^{k}}:\sum_{y\in S_0}d_y-\sum_{y\in S_1}d_y\le \zeta\right\}.\]
Then, $|A^{(\zeta)}|=\sum_{y\in S_0}|W_y|{\wh {\mc L}}_y(H_\zeta)$. Let us set \[f(\zeta)=\sum_{y\in S_0}\pi_y\mc L_y(H_\zeta).\] 
By the second point in \cref{lem:halfspace}, the function $f$ satisfies a Lipschitz-like property: if $|\zeta-\zeta'|\le n^{-\Omega(1)}$ then $|f(\zeta)-f(\zeta')|\le n^{-\Omega(1)}$. Since $\lim_{\zeta\to -\infty}f(\zeta)=0$ and $\lim_{\zeta\to \infty}f(\zeta)=\sum_{y\in S_0}\pi_y=1/2+o(1)$, there is some $\zeta_\alpha$ such that $|f(\zeta_\alpha)-\alpha|\le n^{-\Omega(1)}$.

By the first point in \cref{lem:halfspace}, we then have $||A'|-|A^{(\zeta_\alpha)}||\le n^{1-\Omega(1)}$. That is to say, the set $A'$ differs from the set $A^{(\zeta_\alpha)}$ by only $n^{1-\Omega(1)}$ elements (noting that either $A'\subseteq A^{(\zeta)}$ or $A^{(\zeta)}\subseteq A'$ always). Again using the first point in \cref{lem:halfspace}, it follows that 
\[|W_{z}'|=|W_z\cap A'|=|W_z\cap A^{(\zeta_\alpha)}|\pm n^{1-\Omega(1)}=|W_z|\wh{\mc L}(H_{\zeta_\alpha})\pm n^{1-\Omega(1)}=\pi_z'n\pm n^{1-\Omega(1)},\]
as desired, where $\pi_z'=\pi_z\mc L(H_{\zeta_\alpha})$.

\subsubsection{Some intermediate empirical degree distributions}\label{subsec:intermediate}
For a vertex $v$, define the degree vector
\begin{equation}\vec{g}(v)=\left((\deg_{W_y}(v)-|W_y|/2)/\sqrt{n}\right)_{y\in \{0,1\}^{k}}\label{eq:uv}\end{equation}
(which is determined by $\mc R$), and recall that for $z\in \{0,1\}^{k}$, $\wh{\mc L}_z$ is the uniform measure on the sequence $(\vec g(v))_{v\in W_z}$. For $b\in \{0,1\}$, let $\wh{\mc D}_{(z,b)}$ be the uniform measure on $(\vec g(v))_{v\in V_{(z,b)}}$ (which depends on $\mc R,\mc S$, but not the remaining randomness of $G$). This can be thought of as an `intermediate' empirical degree distribution between $\wh{\mc L}_z$ and $\wh{\mc L}_{(z,b)}$, where we consider the degrees from vertices in $V_{(z,b)}$ into the sets $W_y$.

The considerations in the previous section give us quite strong control over the $\wh{\mc D}_{(z,b)}$. Indeed, for any box $B\subseteq \RR^{\{0,1\}^k}$ let $W_{z}(B)$ be the set of all $v\in W_z$ with $\vec g(v)\in B$, and as in the last section, assume without loss of generality that $z_k=0$. Let $\rho_z'(B)=\pi_z\mc L_z(B\cap H_{\zeta_\alpha})$, so that $|W_z(B)\cap W_z'|=\rho_z'(B)n \pm n^{1-\Omega(1)}$, and a concentration inequality for the hypergeometric distribution shows that with probability $1-O(1/n)$ over the randomness of $\mc S$, we have
$|W_z(B)\cap V_{(z,b)}|=\rho_z(B)n \pm n^{1-\Omega(1)}$, where
\begin{align*}\rho_{(z,b)}(B)&=\begin{cases}
\alpha^4\rho_{z}'(B)+ (1-\alpha^4)(\pi_z\mc L_z(B)-\rho_{z}'(B))&\text{if $z_k=b$},\\
\alpha^4(\pi_z\mc L_z(B)-\rho_{z}'(B))+(1-\alpha^4)\rho_{z}'(B)&\text{if $z_k\ne b$}.
\end{cases}
\end{align*}
Since $\wh{\mc D}_{(z,b)}(B)=|W_z(B)\cap V_{(z,b)}|/|V_{(z,b)}|$, \cref{B1} implies that $\wh{\mc D}_{(z,b)}(B)=\mc D_{(z,b)}(B) \pm n^{-\Omega(1)}$, where $\mc D_{(z,b)}$ is the probability distribution for which $\mc D_{(z,b)}(S)$ is proportional to $\rho_{(z,b)}(S)$ for all boxes $S\subseteq \RR^{\{0,1\}^k}$. Recalling \cref{C3} and \cref{C4}, and partitioning the big box 
\[Q=\left(-C_{\alpha,k-1}\sqrt{\log n},C_{\alpha,k-1}\sqrt{\log n}\right]^{2^k}\]
into $n^{c/2+o(1)}$ boxes with side lengths $n^{-c/(2\cdot 2^k)}$ for a sufficiently small $c > 0$, it follows from \cref{lem:box-kolmogorov} that $\operatorname d_{\mathrm{K}}(\wh{\mc D}_{(z,b)},\mc D_{(z,b)}) \le n^{-\Omega(1)}$ with high probability over the randomness of $\mc S$.

\subsubsection{Controlling the outlier degrees}\label{subsec:outliers}
We next prove that \cref{B3} holds with high probability. In addition to our conditioning on $\mc R$, in this subsection we also condition on an outcome of $\mc S$ such that each $|V_x|=\Omega(n)$ (we have just observed in our consideration of \cref{B1} that such bounds hold with high probability).

Fix an arbitrary $x\in \{0,1\}^{k+1}$ and $y\in \{0,1\}^k$. We wish to show that with high probability, for every $v\in W_y$ we have $\left|{\deg_{V_x}(v)}-|V_x|/2\right|\le C_{\alpha,k}\sqrt{n\log n}$, for some $C_{\alpha,k}>0$. This suffices, since we will then be able to take the union bound over all $O(1)$ choices of $x,y$. The desired bound follows from part~\eqref{enumcons1} of \cref{lem:enum-consequence} and part~\eqref{enumconsbip1} of \cref{lem:enum-consequence-bipartite} along with a Chernoff bound for the hypergeometric distribution and a union bound over $v \in W_y$: if $z=(x_1,\dots,x_k)$ satisfies $z=y$, then we consider the degree-constrained random graph $G[W_y]$, and if we instead have $z\ne y$, then we consider the degree-constrained bipartite graph $G[W_y,W_z]$.

\subsubsection{Defining the ideal distributions}
We shall address \cref{B4} first before turning to \cref{B2} (which is by far the most involved of the four properties). Therefore, at this juncture, we take a moment to say something about how we will define the distributions $\mc{L}_x$ for $x\in \{0,1\}^{k+1}$. First, for specific outcomes of $\mc R,\mc S$ (which determine the sets $V_x$ for $x\in\{0,1\}^{k+1}$), we let $\mc L_x^{\mc R,\mc S}$ be the distribution obtained by choosing a random $v\in V_x$ and sampling its degree vector
\[\vec{d}(v)=\left((\deg_{V_y}(v)-|V_y|/2)/\sqrt{n}\right)_{y\in \{0,1\}^{k+1}}\]
according to the remaining randomness in $G$. We will later show that if $\mc R$ is well-behaved, and $\mc S$ also satisfies certain properties that hold with high probability, then $\mc L_x^{\mc R,\mc S}$ is actually not very sensitive to the specific choice of $\mc R$ and $\mc S$, whence we will be able to prove that \cref{B2} holds with high probability when we take $\mc L_x$ to be any such $\mc L_x^{\mc R,\mc S}$.

\subsubsection{Anti-concentration}\label{subsec:anticoncentration}
Here, we show that \cref{B4} holds. As in \cref{subsec:outliers}, we condition on a well-behaved outcome of $\mc R$ as well as on an outcome of $\mc S$ such that each $|V_x|=\Omega(n)$. By the above discussion, it suffices to show that $\mc L_x^{\mc R,\mc S}$ satisfies the anti-concentration property in \cref{B4}. The rough idea for establishing this involves combining \cref{lem:enum-consequence,lem:enum-consequence-bipartite} (which provide anti-concentration subject to the remaining randomness in $G$) with the anti-concentration property in \cref{C4} coming from the outcome of the process so far.

Fix a vertex $v\in W_z$ for some $z \in \{0,1\}^k$. By part~\eqref{enumcons1} of \cref{lem:enum-consequence} and part~\eqref{enumconsbip1} of \cref{lem:enum-consequence-bipartite}, for $y\in \{0,1\}^k$ and $t\in \NN$, parameterising $t=|V_{(y,0)}|/2+\tau\sqrt n$ and writing $d_v = \deg_{W_y}(v)$, we have
\[
\Pr\left(\deg_{V_{(y,0)}}(v) = t\right) \le \exp\left(O\left(\sqrt{\log n}\right)\right)n^{-1/2}
\]
uniformly in $t$. Indeed, when applying \cref{lem:enum-consequence}, this holds with room to spare when $|\tau| > |V_{(y,0)}|^{1/10} = \Omega(n^{1/10})$, and when $|\tau| \le |V_{(y,0)}|^{1/10}$, we may see that we uniformly have
\begin{align*}
\Pr\left(\deg_{V_{(y,0)}}(v) = t\right) &\le\exp\left(O\left(|\tau| +\sqrt{\log n}\right)\right) \frac{\binom{|V_{(y,0)}|}{t}\binom{|V^c_{(y,0)}|}{d_v-t}}{\binom{m-1}{d_v}}\\
&\le \exp\left(O\left(\sqrt{\log n}\right)\right)n^{-1/2}
\end{align*}
by a standard anti-concentration inequality for the hypergeometric distribution (see for example \cite[Lemma~3.2]{FKS20}).

Since we are conditioning on $\mc R,\mc S$, the degree-constrained random graph $G[W_z]$ and the degree-constrained bipartite graphs $G[W_z,W_y]$ are all independent, so the $2^k$ different degrees $\deg_{V_{(y,0)}}(v)$, for $y\in \{0,1\}^{k}$, are all independent as well. Thus, we obtain the uniform joint anti-concentration bound
\[
\Pr\left(\deg_{V_{(y,0)}}(v) = t_y \text{ for all } y\in\{0,1\}^k\right) \le \exp\left(O\left(\sqrt{\log n}\right)\right) \left(n^{-1/2}\right)^{2^k}.
\]
Note that for each $y  \in \{0,1\}^k$, the degrees $\deg_{V_{(y,0)}}(v)$ and $\deg_{V_{(y,1)}}(v)$ are certainly not independent, since $\deg_{V_{(y,0)}}(v) + \deg_{V_{(y,1)}}(v) = \deg_{W_y}(v)$ is determined by $\mc R$. Nonetheless, our joint anti-concentration bound does imply that for any box $B\subseteq \RR^{\{0,1\}^{k+1}}$ with side lengths $D\ge 1/\sqrt n$, we have
\begin{equation}
\Pr\left(\vec{d}(v)\in B\right) \le \exp\left(O\left(\sqrt{\log n}\right)\right) D^{2^k}.\label{eq:joint-anticoncentration}
\end{equation}
Note that $\vol(B)=D^{2^{k+1}}$, so \cref{eq:joint-anticoncentration} only provides `half as much anti-concentration' as we desire for \cref{B4}. So far, we have only considered anti-concentration of $\vec d(v)$ when $v$ is a fixed vertex; we will next establish the remainder of our anti-concentration and \cref{B4} proper by allowing $v$ to vary and appealing to \cref{C2,C4}.

Recall the definition of the degree vectors $\vec g(v)$ and the empirical distributions $\wh{\mc D}_{(z,b)}$ from \cref{subsec:intermediate}. Each $\wh{\mc D}_{(z,b)}$ is obtained from $\wh{\mc L}_{z}$ by conditioning on an event that holds with probability $\Omega(1)$, so \cref{C4} implies the same anti-concentration property for these intermediate empirical distributions, i.e.,
\begin{equation}
    \wh{\mc D}_x(B) \le \exp\left(O\left(\sqrt{\log n}\right)\right)\vol(B).\label{eq:Dx-anticoncentration}
\end{equation}
for all boxes $B\subseteq\RR^{\{0,1\}^{k+1}}$ with side lengths at least $n^{-c}$, where $c=c_{\alpha,k-1}$, and all $x \in \{0,1\}^{k+1}$.

Now, let $\pi:\RR^{\{0,1\}^{k+1}}\to \RR^{\{0,1\}^k}$ be the projection map $(d_x)_{x\in \{0,1\}^{k+1}}\mapsto (d_{(y,0)}+d_{(y,1)})_{y\in \{0,1\}^k}$. Note that $\vec g(v)=\pi(\vec d(v))$ for all $v$, and note that if $B\subseteq \RR^{\{0,1\}^{k+1}}$ is a box with side lengths $n^{-c}$, then $\pi(B)$ is contained in a box with side lengths $2n^{-c}$. So, by \cref{eq:joint-anticoncentration,eq:Dx-anticoncentration}, we have
\begin{align*}
    {\mc L}_x^{\mc R,\mc S}(B)&=\sum_{v\in V_x:\vec g(v)\in \pi(B)} \frac{1}{|V_x|}\cdot \Pr\left(\vec d(v)\in B\right)\\
    &\le \wh{\mc D}_x(\pi(B))\sup_{v\in V_x}\Pr\left(\vec d(v)\in B\right)\\
    &\le  \exp\left(O\left(\sqrt{\log n}\right)\right) (2n^{-c})^{2^k} (n^{-c})^{2^k}\\
    &\le \exp\left(O\left(\sqrt{\log n}\right)\right) \vol(B)
\end{align*}
for all $x \in \{0,1\}^{k+1}$, as desired.

\subsubsection{Concentration of the empirical degree distributions}
In this subsection we use a second moment calculation as in \cref{sub:base-case} to show that, if we condition on appropriate outcomes of $\mc R$ and $\mc S$, then with high probability, for any $x\in \{0,1\}^{{k+1}}$, we have 
\[\operatorname{d}_{\mathrm K}\left(\wh{\mc L}_x,\mc L_x^{\mc R,\mc S}\right) \le  n^{-\Omega(1)}.\]
We shall later prove that the distributions $\mc L_x^{\mc R,\mc S}$, for appropriate $\mc R,\mc S$, are all Kolmogorov-close to each other; it will then follow that \cref{B2} holds with high probability.

As in the previous two subsections, we condition on a well-behaved outcome of $\mc R$ and an outcome of $\mc S$ for which $|V_x|=\Omega(n)$ for all $x \in \{0,1\}^{k+1}$. Fix an $x\in \{0,1\}^{{k+1}}$, and as before, let $Q=(-C_{\alpha,k}\sqrt{\log n},C_{\alpha,k}\sqrt{\log n}]^{2^{k+1}}$, where $C_{\alpha,k}$ is as chosen in \cref{subsec:outliers} (so, we have say $\mc L_x(Q^c)\le n^{-2}$).

We wish to apply \cref{lem:empirical-boxes}. To this end, we shall, for an arbitrary pair of vertices $u$ and $v$, study conditional probabilities of the form
\[\Pr\left(\deg_{V_{(z,b)}}(v)=t\,\middle|\,N_{W_z}(u)=T\right),\]
where $z\in \{0,1\}^{{k}}$, $b\in \{0,1\}$, and $T$ is a set of $\deg_{W_z}(u)$ elements of $W_z\setminus \{u\}$. Let $R_{(z,b)}=\left\{t:\left|t-|V_{(z,b)}|/2\right|\le C_{\alpha,k}\sqrt{n \log n}\right\}$. We will show that for such data $u,v,z,b$, and each $t\in R_{(z,b)}$, the value of the above conditional probability is not very sensitive to the choice of $T$.

Let $y\in \{0,1\}^k$ be such that $v\in W_y$. As usual, we need to consider separately the case where $y=z$ and where $y\ne z$; in the former case, we study the degree-constrained random graph $G[W_y]$, and in the latter case we study the degree-constrained random bipartite graph $G[W_y,W_z]$. 

If $y=z$, then having conditioned on the event $N_{W_y}(u)=T$, now $G[W_y\setminus\{u\}]$ is a random graph with a particular degree sequence (namely, the degree sequence where we delete $u$ if it is in $W_y$, and if so we also decrement the degree of every vertex in $T$ by one). Considering how this degree sequence varies for different choices of $T,T'$, it follows from part~\eqref{enumcons2} of \cref{lem:enum-consequence} (and the first part of \cref{def:close}) that for each $u,v,z,b$ as above, each $t\in R_{(z,b)}$, and each such pair $T,T'$, we have \[\Pr\left(\deg_{V_{(z,b)}}(v)=t\,\middle|\,N_{W_z}(u)=T\right)\simeq \Pr\left(\deg_{V_{(z,b)}}(v)=t\,\middle|\,N_{W_z}(u)=T'\right).\]
We obtain the same conclusion if $z\ne y$ by considering the bipartite graph $G[W_y,W_z]$, except now relying on \cref{lem:enum-consequence-bipartite}. 

The above argument implies that for all $u,v,z,b,t$ as above, we in fact have
\[\Pr\left(\deg_{V_{(z,b)}}(v)=t\,\middle|\,N_{W_z}(u)=T\right)\simeq\Pr\left(\deg_{V_{(z,b)}}(v)=t\right).\]
Observing that all the random subgraphs of the form $G[W_y],G[W_y,W_z]$ are independent, we deduce that for any $\vec \tau,\vec \sigma\in Q$, we have 
\[\Pr(\vec d(v)=\vec \tau \text{ and } \vec d(u)=\vec \sigma)\simeq\Pr(\vec d(v)=\vec \tau)\Pr(\vec d(u)=\vec \sigma).\]
Therefore we can apply \cref{lem:empirical-boxes}, using \cref{B4} (which we have already proved) and the fact that $\mc L_x(Q^c)\le 1/n^2$ for all $x \in \{0,1\}^{k+1}$, to conclude that \cref{B2} holds with high probability.

\subsubsection{Sensitivity to the conditioned information}\label{subsec:sensitivity}
To finish, we wish to show that for all $x\in \{0,1\}^{{k+1}}$, well-behaved $\mc R$ and $\mc R'$, and almost all outcomes $\mc S$ and $\mc S'$, we have \[\operatorname{d}_{\mathrm K}\left(\mc L_x^{\mc R,\mc S},\mc L_x^{\mc R',\mc S'}\right) \le  n^{-\Omega(1)}.\]
This will complete the proof of the inductive step of \cref{prop:empirical-induction}. 

Recall the definitions of the degree vectors $\vec g(v)$ and the intermediate degree distributions $\wh{\mc D}_x$, $\mc D_x$ from \cref{subsec:intermediate}. In that subsection, we showed for all well-behaved $\mc R$ that, with high probability over $\mc S$, we have
$\operatorname{d}_{\mathrm K}(\wh{\mc D}_{x},\mc D_{x})=n^{-\Omega(1)}$. Let $c$ (depending on $\alpha,k$) be sufficiently small such that $\operatorname{d}_{\mathrm K}(\wh{\mc D}_{x},\mc D_{x})\le n^{-c}$ with high probability, and let us now call an outcome of $\mc S$ \emph{well-behaved} if this is the case for all $x\in \{0,1\}^{{k+1}}$.

Let $\pi:\RR^{\{0,1\}^{k+1}}\to \RR^{\{0,1\}^{k}}$ be the projection map $(d_x)_{x\in \{0,1\}^{k+1}}\mapsto (d_{(y,0)}+d_{(y,1)})_{y\in \{0,1\}^k}$, as was the case in \cref{subsec:anticoncentration}. If we condition on any $\mc R,\mc S$, then for any $v\in V_x$ and any $\vec \tau\in \RR^{\{0,1\}^{k+1}}$ with $\vec g(v)=\pi(\vec \tau)$, we have
\[\Pr(\vec d(v)=\vec \tau)=\prod_{y\in \{0,1\}^{k}}\Pr(\deg_{V_{(y,0)}}(v)=t_y),\]
where $(t_y-|V_{(y,0)}|/2)/\sqrt n=\tau_y$. Now, probabilities of the form $\Pr(\deg_{V_{x}}(v)=t)$ are actually not very sensitive to the specific choice of $v,t,\mc R,\mc S$, in the following sense. Suppose $\mc R,\mc S, \mc R',\mc S'$ are all well-behaved, and for some $y\in \{0,1\}^{k}$, let $v\in W_y^{\mc R}$ and $v'\in W_y^{\mc R'}$ be vertices in the `same part' with respect to $\mc R$ and $\mc R'$. Moreover, suppose that 
\[\left\|\vec g^{\mc R}(v)-\vec g^{\mc R'}(v')\right\|_\infty \le n^{1/2-\Omega(1)}.\] Then for any $x\in \{0,1\}^{{k+1}}$ and $t,t'=\pi_x n / 2\pm n^{1/2-\Omega(1)}$, by part~\eqref{enumcons2} of \cref{lem:enum-consequence} and part~\eqref{enumconsbip2} of \cref{lem:enum-consequence-bipartite} (and using the second part of \cref{def:close}), we have
\begin{equation}\Pr\left(\deg_{V_{x}^{\mc R,\mc S}}(v)=t\,\middle |\,\mc R,\mc S\right)\simeq \Pr\left(\deg_{V_{x}^{\mc R',\mc S'}}(v')=t'\,\middle |\,\mc R',\mc S'\right).\label{eq:insensitive-complicated}\end{equation}

Now, consider well-behaved data $\mc R,\mc S, \mc R',\mc S'$, and fix some $x\in \{0,1\}^{{k+1}}$. Our next objective is to construct an injective mapping $\phi$ between $V_x^{\mc R,\mc S}$ and $V_x^{\mc R',\mc S'}$ that maps a vertex $v\in V_x^{\mc R,\mc S}$ to a vertex $\phi(v)\in V_x^{\mc R',\mc S'}$ with `roughly the same statistics' as $v$. This will allow us to compare probabilities conditional on the outcomes $(\mc R,\mc S)$ with probabilities conditional on the outcomes $(\mc R',\mc S')$.

First, let $Q=(-C_{\alpha,k}\sqrt{\log n},C_{\alpha,k}\sqrt{\log n}]^{2^{k+1}}$, so by the same considerations as in \cref{subsec:outliers}, we know that 
\[\mc L_x^{\mc R,\mc S}(Q^c)\le 1/n \,\text{ and }\, \mc L_x^{\mc R',\mc S'}(Q^c)\le 1/n.\] 
Now, partition $Q$ into a collection $\mf B$ of $n^{-c/2+o(1)}$ boxes with side lengths $n^{-c/(2\cdot 2^{k+1})}$. Since $\mc R,\mc S, \mc R',\mc S'$ are all well-behaved, we have 
\[\operatorname{d}_{\mathrm K}\left(\wh{\mc D}_{x}^{\mc R,\mc S},\wh{\mc D}_{x}^{\mc R',\mc S'}\right)\le n^{-c}.\] Also, we may assume with no loss of generality that  $c$ is sufficiently small, and in particular, that $c <c_{\alpha,k}$, so by \cref{B1}, we have $|V_x^{\mc R,\mc S}|=|V_x^{\mc R',\mc S'}|\pm n^{1-c}$. It follows that, for each $B\in\mf B$, if we consider the sets 
\[V_{x}^{\mc R,\mc S}(B)=\{v\in V_x^{\mc R,\mc S}:\vec g(v)\in \pi(B)\} \,\text{ and }\, V_{x}^{\mc R',\mc S'}(B)=\{v\in V_x^{\mc R',\mc S'}:\vec g(v)\in \pi(B)\},\]
then we have
\[|V_{x}^{\mc R,\mc S}(B)|=|V_{x}^{\mc R',\mc S'}(B)|\pm O\left(n^{1-c}\right).\] 
Now, let 
\[m(B)= \min \left\{|V_{x}^{\mc R,\mc S}(B)|,|V_{x}^{\mc R',\mc S'}(B)|\right\},\] 
and let $U\subseteq V_{x}^{\mc R,\mc S}$ be obtained by choosing $m(B)$ elements from each $V_{x}^{\mc R,\mc S}(B)$ for $B \in \mf B$, so that 
\[|U| \ge |V_{x}^{\mc R,\mc S}|-O\left(n^{1-c/2+o(1)}\right).\] 
Let $\phi:U\to V_x^{\mc R',\mc S'}$ be an injection such that $\phi(v)\in V_{x}^{\mc R',\mc S'}(B)$ for each $v\in U\cap V_{x}^{\mc R,\mc S}(B)$. Each $B\in \mf B$ has $\ell^\infty$-diameter $O(n^{-c/(2\cdot 2^{k+1})})$, so applying \cref{eq:insensitive-complicated} and summing over points in $B$, we see for all $v\in U$ that
\[\Pr\left(\vec d(v)\in B\,|\,\mc R,\mc S\right )=(1\pm n^{-c'})\Pr\left(\vec d(\phi(v))\in B\,|\,\mc R',\mc S'\right),\]
for some $c'>0$ depending on $c$ and $k$. Now, if we coarsen $\mf B$ into a partition $\mf B'$ of $n^{-c'/2+o(1)}$ boxes with side lengths at most $n^{-c'/(2\cdot 2^{k+1})}$, then we easily see that the conditions of \cref{lem:box-kolmogorov} are satisfied, and we deduce that $\operatorname{d}_{\mathrm K}(\mc L_x^{\mc R,\mc S},\mc L_x^{\mc R',\mc S'}) \le n^{-\Omega(1)}$ as desired. This finishes the inductive proof of \cref{prop:empirical-induction}.

\section*{Acknowledgements}
The first author was supported in part by NSF grants DMS-1954395 and DMS-1953799. The second author was supported by NSF grant DMS-1953990. The third author was supported by NSF grant DMS-180052. The fourth author and fifth author were both supported by NSF Graduate Research Fellowship Program DGE-1745302.
\bibliographystyle{amsplain_initials_nobysame_nomr.bst}
\bibliography{friendly_partitions.bib}

\appendix

\section{Probabilities in degree-constrained graph models}\label{app:MW}
We start by showing how \cref{lem:enum-consequence} follows from a series of results of increasing precision about random graphs with specified degree sequences. 

\begin{proposition}\label{prop:graph-bounded}
Let $(d_w)_{w\in W}$ be a sequence with even sum on a set $W$ of $n$ vertices such that, defining $\beta_w$ by $d_w = (n-1)/2+\beta_w\sqrt{(n-1)} / 2$, we have
\begin{itemize}
\item $|\beta_w| \le \log n$ for each $w \in W$, and
\item $\sum_{w \in W} \beta_w^2 \le n (\log n)^{1/9}$.
\end{itemize} 
Such a sequence is a graphic sequence for all sufficiently large $n$. Let $G$ be a uniformly random graph with this degree sequence on the vertex set $W$. Consider any fixed $v \in W$, any fixed subset $S\subseteq W$ of size $h$ satisfying $\min(h,n-h)\ge n/(\log n)^{1/8}$, and an integer $t\in[0,d_v]$. If $|t-h/2| > n^{3/5}$, then we have
	\begin{equation}\label{gr-bd-big}
	    \mb{P}(\deg_S(v) = t)  \le \exp(-\Omega((t-h/2)^2/n)).
	\end{equation}
If $|t-h/2|\le n^{3/5}$ on the other hand, then we have
        \begin{equation}\label{gr-bd-small}
        \mb{P}(\deg_S(v) = t)
		= (1\pm O(n^{-1/10}))\frac{\binom{h}{t}\binom{n-h-1}{d_v-t}}{\binom{n-1}{d_v}}\exp(\Lambda_1 - \Lambda_2 - \Lambda_3 + \Lambda_4),
		\end{equation}
where $\Lambda_1$, $\Lambda_2$, $\Lambda_3$ and $\Lambda_4$ are given by	
\begin{align*}
		\Lambda_1 &= \frac{1}{2n^2}\left(\sum_{i \in W}\beta_i\right)\left(\sum_{i\in W}\beta_i-2n\beta_v\right),\\
		\Lambda_2 &= \sum_{i\in S\setminus v}\left(1-\frac{2t}{h}\right)\frac{\beta_i}{\sqrt{n-1}}+\sum_{i\in S^c\setminus v}\left(1-\frac{2(d_v-t)}{(n-h)}\right)\frac{\beta_i}{\sqrt{n-1}},\\
		\Lambda_3 &= \frac{1}{2}\sum_{i\in W\setminus v}\frac{\beta_i^2}{n-1}, \text{ and }\\
		\Lambda_4 &= \frac{1}{2nh}\sum_{i, j\in S\setminus v}(\beta_i-\beta_j)^2 + \frac{1}{2n(n-h)}\sum_{i ,j\in S^c\setminus v}(\beta_i-\beta_j)^2,
	\end{align*}
the sums in the definition of $\Lambda_4$ being over all (unordered) two-element subsets.
\end{proposition}

First, we deduce \cref{lem:enum-consequence} from \cref{prop:graph-bounded}. To this end, we need the following lemma comparing the moments of distributions that are bounded and Kolmogorov-close.

\begin{lemma}\label{lem:kolmogorov-moment}
Fix a constant $c>0$. Let $(a_v)_{v\in V}$ and $(b_u)_{u\in U}$ be two sequences of $\Omega(n)$ real numbers with $\left|\vphantom{\deg_{V_x}}|V|-|U|\right|\le n^{1-c}$ satisfying $|a_v|,|b_u|< q$, and such that the uniform measures $\wh{\mc A},\wh{\mc B}$ on the two lists satisfy $\operatorname{d}_{\mathrm K}(\wh{\mc A},\wh{\mc B})\le n^{-c}$. Then, for all $k \in \mb{N}$, we have 
\[\left|\sum_{v\in V}a_v^k-\sum_{u\in U}b_u^k\right|=O(q^kn^{1-c}).\]
\end{lemma}
\begin{proof}
First, note that
\begin{align*}
    \frac{1}{|V|}\sum_{v\in V}a_v^k&=\int_0^{q}kt^{k-1}(1-\wh{\mc A}\left((-\infty,t])\right)\operatorname d t-\int_{-q}^{0}kt^{k-1}\wh{\mc A}((-\infty,t])\operatorname d t\\
    &=\int_0^{q}kt^{k-1}(1-\wh{\mc B}\left((-\infty,t])\right)\operatorname d t-\int_{-q}^{0}kt^{k-1}\wh{\mc B}((-\infty,t])\operatorname d t\pm O(q^kn^{-c})\\
    &=\frac{1}{|U|}\sum_{u\in U}b_u^k \pm O(q^kn^{-c}).
\end{align*}
The desired result now follows from the fact that $|V|=(1\pm O(n^{-c}))|U|$.
\end{proof}

We are now ready for the proof of \cref{lem:enum-consequence}.
\begin{proof}[Proof of \cref{lem:enum-consequence}]
We shall estimate the probabilities in question using \cref{prop:graph-bounded}. Indeed, the hypothesis in the statement of \cref{lem:enum-consequence}, in the language of \cref{prop:graph-bounded}, may be stated as
\begin{itemize}
    \item $|\beta_w| = O(\sqrt{\log n})$ (and hence $|\beta_w| \le \log n$) for each $w \in W$,\
    \item $|\sum_{w\in T} \beta_w| = O(n)$ for all $T \subseteq W$, and 
    \item $\sum_{w \in W} \beta_w^2 = O(n)$,
\end{itemize}
whence it is clear that \cref{prop:graph-bounded} applies.

For part~\eqref{enumcons1} of \cref{lem:enum-consequence}, we may argue as follows. If $|t- h/2| > n^{3/5}$, then \cref{gr-bd-big} gives us what we need. If $|t- h/2| \le n^{3/5}$, we claim that \cref{gr-bd-small} implies the bound in part~\eqref{enumcons1} of \cref{lem:enum-consequence}. To see this, it suffices to verify in this regime that each of  $|\Lambda_1|$, $|\Lambda_2|$, $|\Lambda_3|$ and $|\Lambda_4|$ are $O(|\tau|+\sqrt{\log n})$, where $\tau$ is defined by  $t = h/2 + \tau \sqrt{n}$.

Using the facts that  $|\sum_{i\in W} \beta_i| = O(n)$, and $\sum_{i \in W} \beta_i^2 = O(n)$, we may bound $|\Lambda_1|$ by
\begin{align*}
		|\Lambda_1| &= \left|\frac{1}{2n^2}\left(\sum_{i \in W}\beta_i\right)\left(\sum_{i\in W}\beta_i-2n\beta_v\right)\right|\\
		&\le  \frac{1}{2n^2}\left(\sum_{i \in W}\beta_i\right)^2 + \frac{|\beta_v|}{n}\left|\sum_{i \in W}\beta_i\right|\\
		&= O(1) + O(|\beta_v|) = O\left(\sqrt{\log n}\right).
	\end{align*}
Next, we bound $|\Lambda_2|$ using the facts that $h, n-h  = \Omega(n)$, $|\sum_{i\in S \setminus v} \beta_i| = O(n)$ and $|\sum_{i\in S^c \setminus v} \beta_i| = O(n)$ by 
\begin{align*}
	|\Lambda_2| &\le \left|\sum_{i\in S\setminus v}\left(1-\frac{2t}{h}\right)\frac{\beta_i}{\sqrt{n-1}}\right|+\left|\sum_{i\in S^c\setminus v}\left(1-\frac{2(d_v-t)}{(n-h)}\right)\frac{\beta_i}{\sqrt{n-1}}\right|,\\
		&\le O\left(|\tau| / n\right) \left|\sum_{i\in S\setminus v} \beta_i \right| + O\left(|\tau|/ n + \sqrt{\log n}/ n\right) \left|\sum_{i\in S^c\setminus v} \beta_i \right|\\
		&= O(|\tau| + \sqrt{\log n}).
	\end{align*}
Finally, since $\sum_{i \in W} \beta_i^2 = O(n)$, it is immediate that $|\Lambda_3| = O(1)$, and it follows from the facts that $\sum_{i \in W} \beta_i^2 = O(n)$, $|\sum_{i\in S \setminus v} \beta_i| = O(n)$ and $|\sum_{i\in S^c \setminus v} \beta_i| = O(n)$ that $|\Lambda_4| = O(1)$ as well.

For part~\eqref{enumcons2} of \cref{lem:enum-consequence}, it is sufficient to verify that the expression in~\eqref{gr-bd-small} is polynomially-stable when the parameters in question vary by the amounts specified in the statement of \cref{lem:enum-consequence}; here, we say that an expression is polynomially-stable if it varies by at most a multiplicative factor of $1 \pm n^{-\Omega(1)}$. This may be done term by term, as we outline below.

Suppose $(d'_w)_{w \in W'}$, $|W'| = n'$, $v' \in W'$, $S' \subseteq W'$, $|S'| = h'$ and $t'$ satisfy the hypothesis in the statement of the lemma, and additionally, are such that
    \begin{itemize}
        \item $|t-t'|,|d_v-d_{v'}'|\le n^{1/2-\Omega(1)}$,
        \item $(d_w)_{w\in S}$ and $(d'_w)_{w\in S'}$ are proximate, and
        \item $(d_w)_{w\in W\setminus S}$ and $(d'_w)_{w\in W'\setminus S'}$ are proximate.
        \item $|n - n'|, |h - h'| \le n^{1-\Omega(1)}$, this being a consequence of the previous two points.
    \end{itemize}
In the regime where $h, n-h = \Omega(n)$, $d = n/2 \pm O(\sqrt{n \log n})$, $t = h/2 \pm O(\sqrt{n \log n})$, the expression \[\binom{h}{t}\binom{n-h-1}{d-t}\binom{n-1}{d}^{-1}\] 
is polynomially-stable when $n$ and $h$ vary by $n^{1-\Omega(1)}$, and $d$ and $t$ vary by $n^{1/2-\Omega(1)}$, which in particular tells us that
\[\binom{h}{t}\binom{n-h-1}{d_v-t}\binom{n-1}{d_v}^{-1} \simeq \binom{h'}{t'}\binom{n'-h'-1}{d'_{v'}-t'}\binom{n'-1}{d'_{v'}}^{-1}.\] 
This can be seen via a careful (and rather tedious) application of Stirling's approximation, or alternately, by using a sufficiently precise form of the de Moivre–-Laplace normal approximation, as in~\cite{ZS13} for example.

Next, we need to verify that each of $\exp(\Lambda_1)$, $\exp(-\Lambda_2)$, $\exp(-\Lambda_3)$ and $\exp(\Lambda_4)$ are similarly polynomially-stable, and this may be accomplished in a straightforward manner using \cref{lem:kolmogorov-moment}. To illustrate, we spell out the details for $\exp(-\Lambda_3)$ below.

Recall that
\[	\Lambda_3 = \frac{1}{2}\sum_{i\in W\setminus v}\frac{\beta_i^2}{n-1} = \frac{1}{2}\sum_{i\in W}\frac{\beta_i^2}{n-1} \pm O(\log n/n).\]
Our goal is to show, with $\beta'_i$ defined by $d'_i = (n'-1)/2+\beta'_i\sqrt{(n'-1)} / 2$ for $i \in W'$, that 
\[	\Lambda'_3 = \frac{1}{2}\sum_{i\in W'\setminus v'}\frac{(\beta'_i)^2}{n'-1} = \frac{1}{2}\sum_{i\in W'}\frac{(\beta'_i)^2}{n'-1} \pm O(\log n/n)\]
is close enough to $\Lambda_3$ to ensure $\exp(-\Lambda_3) \simeq \exp(-\Lambda'_3)$. 

Since $(d_w)_{w\in S}$ and $(d_w')_{w\in S'}$ are proximate, we claim that
\[ \left|\sum_{i \in S} \beta^2_i - \sum_{i \in S'}(\beta'_i)^2\right| \le n^{1-\Omega(1)}.\]
This is true with room to spare if the two sequences are proximate on account of the first part of \cref{def:close}, since in this case, we know that 
\[ \sum_{i \in S} \left|\beta_i - \beta'_{\psi(i)}\right| = O(\sqrt{n})\]
for some bijection $\psi:S \to S'$, from which it follows that 
\[ \left| \sum_{i \in S} \beta^2_i - \sum_{i \in S'}(\beta'_i)^2\right| \le \left(\max_{i \in S} \left|\beta_i + \beta'_{\psi(i)}\right|\right) \left( \sum_{i \in S} \left|\beta_i - \beta'_{\psi(i)}\right|\right) = O(\sqrt{n \log n }).\]
If the two sequences are proximate on account of the second part of \cref{def:close}, then since $|n - n'| \le n^{1-\Omega(1)}$, it is easily checked that the Kolmogorov distance between the uniform measures on $(\beta_i)_{i\in S}$ and $(\beta'_i)_{i\in S'}$ is at most $n^{-\Omega(1)}$, so by \cref{lem:kolmogorov-moment} (with $k = 2$ and $q = \log n$), we have 
\[ \left|\sum_{i \in S} \beta^2_i - \sum_{i \in S'}(\beta'_i)^2\right| \le n^{1-\Omega(1)}\]
as claimed. Reasoning similarly about the proximate pair $(d_w)_{w\in W\setminus S}$ and $(d_w')_{w\in W'\setminus S'}$, we deduce that
\[ \left|\sum_{i \in W\setminus S} \beta^2_i - \sum_{i \in W'\setminus S'}(\beta'_i)^2\right| \le n^{1-\Omega(1)}\]
as well. Putting these pair of estimates together shows that $|\Lambda_3 - \Lambda'_3| \le n^{-\Omega(1)}$, whence it is clear that $\exp(-\Lambda_3) \simeq \exp(-\Lambda'_3)$.

The details in the other three cases (i.e., for $\Lambda_1$, $\Lambda_2$ and $\Lambda_4$) are similar, and we leave them to the reader.
\end{proof}

\cref{prop:graph-bounded} is a consequence of the following more general statement, the proof of which will be given in \cref{app:Proof} once we have collected the requisite machinery in \cref{app:enumresults}.

\begin{proposition}\label{prop:graph-expectation}
	Let $(d_w)_{w\in W}$ be a sequence with even sum on a set $W$ of $n$ vertices such that, defining $\beta_w$ by $d_w = (n-1)/2+\beta_w\sqrt{(n-1)} / 2$, we have  $|\beta_w| \le \log n$ for each $w \in W$. Such a sequence is a graphic sequence for all sufficiently large $n$. Let $G$ be a uniformly random graph with this degree sequence on the vertex set $W$. For any fixed $v \in W$, $S\subseteq W$ of size $h$ satisfying $\min(h,n-h)\ge n/(\log n)^{1/8}$, and an integer $t\in[0,d_v]$, we have
	\[
		\mb{P}(\deg_S(v) = t)=(1\pm O(n^{-1/6}))\frac{\binom{h-\mbm{1}_S(v)}{t}\binom{n-h-\mbm{1}_{S^c}(v)}{d_v-t}}{\binom{n-1}{d_v}} \exp(\Lambda_1 - \Lambda_3)
		\mb{E}_T\left[\exp(-\Lambda_T)\right],
	\]
	where $T = T_1 \cup T_2$ is a random set chosen by picking $T_1$ uniformly from $\binom{S\setminus v}{t}$ and $T_2$ uniformly from $\binom{S^c\setminus v}{d_v-t}$, and where $\Lambda_1$, $\Lambda_3$ and $\Lambda_T$ are given by	
\begin{align*}
		\Lambda_1 &= \frac{1}{2n^2}\left(\sum_{i \in W}\beta_i\right)\left(\sum_{i\in W}\beta_i-2n\beta_v\right),\\
		\Lambda_3 &= \frac{1}{2}\sum_{i\in W\setminus v}\frac{\beta_i^2}{n-1}, \text{ and }\\
    	\Lambda_T &= \sum_{i\in W\setminus v}(-1)^{\mbm{1}_T(i)}\frac{\beta_i}{\sqrt{n-1}}.
	\end{align*}
\end{proposition}

 To proceed, we will need to understand expressions as appearing in the right side of \cref{prop:graph-expectation}. To this end, we state two general results about sums of random variables constrained to live on a ``slice'' of the Boolean hypercube.
\begin{lemma}\label{lem:slice-subgaussian}
Let $a_1,\ldots,a_n\in\mb{R}$ and let $X = \sum_{i=1}^na_i\xi_i$, where $\xi = (\xi_1,\ldots,\xi_n)$ is uniform on the subset of binary vectors in $\{0,1\}^n$ which have sum $s$. Writing $\eta^2 = \sum_{i=1}^na_i^2-(\sum_{i=1}^na_i)^2/n$, we have
\[\mb{P}(|X-\mb{E}[X]|\ge t)\le 2\exp(-t^2/(4\eta^2))\]
and
\[\mb{E}\left[e^{X-\mb{E}[X]}\right]\le 2e^{O(\eta^2)}.\]
\end{lemma}
\begin{proof}
The first part follows from the Azuma--Hoeffding inequality, as outlined in~\cite[Lemma~2.2]{KST19}, for example. The second part follows from integrating the first; see~\cite[Proposition~2.5.2]{Ver18}.
\end{proof}
\begin{lemma}\label{lem:slice-moment}
Let $a_1,\ldots,a_n\in\mb{R}$ and let $X = \sum_{i=1}^na_i\xi_i$, where $\xi = (\xi_1,\ldots,\xi_n)$ is uniform on the subset of $\{0,1\}^n$ with sum $s$ such that $\min(s,n-s)\ge n(\log n)^{-2}$. Suppose that $|a_i|\le n^{-1/2}(\log n)^2$ and $\eta^2 = \sum_{i=1}^na_i^2-(\sum_{i=1}^na_i)^2/n\le \sqrt{\log n}$. Then we have
\[ \mb{E}\left[e^X\right] = \exp\left(\mb{E}[X] + \frac{1}{2}\on{Var}[X] \pm O(n^{-1/9})\right).\]
\end{lemma}
\begin{proof}
Writing $\sigma^2 = \on{Var}[X]$, we clearly have
\begin{align*}
\sigma^2&=\sum_{i\neq j}a_ia_j(\mb{E}[\xi_i\xi_j]-\mb{E}[\xi_i]\mb{E}[\xi_j]) + \sum_i a_i^2(\mb{E}[\xi_i^2]-\mb{E}[\xi_i]^2)\\
&=\sum_{i\neq j}a_ia_j\left(\frac{s(s-1)}{n(n-1)}-\frac{s^2}{n^2}\right) + \sum_i a_i^2\left(\frac{s}{n}-\frac{s^2}{n^2}\right)=\frac{s(n-s)}{n(n-1)}\eta^2.
\end{align*}

First, by \cref{lem:slice-subgaussian}, we have
\[\mb{P}[|X-\mb{E}[X]|\ge t]\le 2\exp(-t^2/(4\eta^2))\]
for all $t\ge 0$. Now
\begin{align*}
\mb{E}\left[e^{X-\mb{E}[X]}\right] &= \int_{-\infty}^\infty e^t\mb{P}(X-\mb{E}[X]\ge t)dt\\
&=\int_{-\infty}^{8\eta\sqrt{\log n}}e^t\mb{P}(X-\mb{E}[X]\ge t)dt +O\left(\int_{8\eta\sqrt{\log n}}^\infty e^{t-t^2/(4\eta^2)}dt\right)\\
&=\int_{-\infty}^{8\eta\sqrt{\log n}}e^t\mb{P}(X-\mb{E}[X]\ge t)dt +O\left(\int_{8\eta\sqrt{\log n}}^\infty e^{-t^2/(8\eta^2)}dt\right)\\
&=\int_{-\infty}^{8\eta\sqrt{\log n}}e^t\mb{P}(X-\mb{E}[X]\ge t)dt+ O(n^{-4}).
\end{align*}
If $\sigma\le n^{-1/8}$, then $\eta$ is similarly bounded and we obtain an upper bound of the form $1+O(n^{-1/9})$. Combining with $\mb{E}e^X\ge e^{\mb{E}X}$, the result follows. If $\sigma > n^{-1/8}$, then a combinatorial central limit theorem of Bolthausen~\cite{Bol84} shows that
\[\on{d}_{\mr{K}}(X-\mb{E}[X],\mc{N}(0,\sigma^2)) = O\left(\sum_{i=1}^n |a_i|^3/\sigma^3\right) = O(n^{-2/17}).\]
This allows us the replace the integrand above with the cumulative distribution function of a Gaussian, and we easily derive
\[\mb{E}\left[e^{X-\mb{E}[X]}\right] = e^{\frac{\sigma^2}{2}} \pm O\left(n^{-2/17} e^{\eta\sqrt{\log n}}\right) = \exp(\sigma^2/2 \pm O(n^{-1/9})).\qedhere\]
\end{proof}

\cref{prop:graph-bounded} is now easily deduced from \cref{prop:graph-expectation}.
\begin{proof}[Proof of \cref{prop:graph-bounded}]
	With $T = T_1 \cup T_2$ a random set chosen by picking $T_1$ uniformly from $\binom{S\setminus v}{t}$ and $T_2$ uniformly from $\binom{S^c\setminus v}{d_v-t}$, we have
	\begin{align*}
		\mb{E}_T[\Lambda_T] &= \mb{E}_T \left[\sum_{i\in W\setminus v}(-1)^{\mbm{1}_T(i)}\frac{\beta_i}{\sqrt{n-1}}\right]\\
		&= \sum_{i\in S\setminus v}\left(1-\frac{2t}{h}\right)\frac{\beta_i}{\sqrt{n-1}}+\sum_{i\in S^c\setminus v}\left(1-\frac{2(d_v-t)}{(n-h)}\right)\frac{\beta_i}{\sqrt{n-1}} \pm O(n^{-1/3})\\
		&= \Lambda_2 \pm O(n^{-1/3}),
	\end{align*}
	where the small additive error term comes from the fact that whether $v\in S$ or $v\in S^c$ slightly change the fractions listed above, but not by much.
	
	At this point, if $|t-h/2| > n^{3/5}$, we have
	\[\frac{\binom{h-\mbm{1}_S(v)}{t}\binom{n-h-\mbm{1}_{S^c}(v)}{d_v-t}}{\binom{n-1}{d_v}}\le\exp(-\Omega((t-h/2)^2/n))\]
	by a standard tail bound for the hypergeometric distribution (see~\cite[Theorem~2.10]{JLR00}, for example). Since $|\beta_w| \le \log n$ for each $w \in W$, clearly both $|\Lambda_1|$ and $|\Lambda_3|$ are $O((\log n)^2)$, whence $\exp(\Lambda_1 - \Lambda_3) \le \exp(O((\log n)^4))$, and we are left with estimating $\mb{E}[\exp(-\Lambda_T)]$. Now \cref{lem:slice-subgaussian} demonstrates \[\mb{E}[\exp(-\Lambda_T)]\le\exp(\mb{E}[-\Lambda_T] + O((\log n)^2)),\] 
	since the coefficient variance in $-\Lambda_T$ is $O((\log n)^2/n)$ by the given conditions. The above explicit expression for $\mb{E}[\Lambda_T]$ demonstrates that
	\[|\mb{E}[-\Lambda_T]| = O\left(\frac{|t-h/2|(\log n)^2}{\sqrt{n}}\right)\]
    when $|t-h/2| > n^{3/5}$. These estimates together immediately yield a bound of the claimed quality.
	
	From now on we assume $|t-h/2|\le n^{3/5}$. We next compute the variance of $\Lambda_T$. Following the computation in the proof of \cref{lem:slice-moment}, we see
	\begin{align*}
		\on{Var}[\Lambda_T] &=\frac{4}{(n-1)}\bigg(\frac{t(h-t)}{h(h-1)}\frac{\sum_{i, j\in S\setminus v}(\beta_i-\beta_j)^2}{h}\\
		&\qquad\qquad\qquad\qquad+ \frac{(d_v-t)((n-h)-(d_v-t))}{(n-h)(n-h-1)}\frac{\sum_{i,j\in S^c\setminus v}(\beta_i-\beta_j)^2}{n-h}\bigg) + O(n^{-1/4}),
	\end{align*}
	these sums being over all (unordered) two-element subsets; here, we again use the fact that the fraction $t/|S\setminus v|$ is close to $t/h$ regardless of if $v\in S$ or $v \in S^c$. Now using $t = h/2 \pm n^{3/5}$ and $d_v = n/2 + O(\sqrt{n}(\log n))$, we find
	\begin{align*}
	\on{Var}[\Lambda_T] &=  \frac{1}{nh}\sum_{i, j\in S\setminus v}(\beta_i-\beta_j)^2 + \frac{1}{n(n-h)}\sum_{i ,j\in S^c\setminus v}(\beta_i-\beta_j)^2 \pm O(n^{-1/4})\\
	&= 2\Lambda_4 \pm O(n^{-1/4}).
	\end{align*}
	Note that $\on{Var}[\Lambda_T]\le\sum_{i=1}^n\beta_i^2/\min(h,n-h) = O(n(\log n)^{1/9}/\min(h,n-h))$, and apply \cref{lem:slice-moment} to the two slices defining $\Lambda_T$. Note that the condition $\eta^2\le\sqrt{\log n}$ follows from the inequalities $(n/h)(\log n)^{1/9} < \sqrt{\log n}$ and the relation between $\sigma^2$ and $\eta^2$ in the proof of \cref{lem:slice-moment}. Therefore
	\begin{align*}\mb{E}[\exp(-\Lambda_T)] &= \exp\left(\mb{E}[-\Lambda_T] + \frac{1}{2}\on{Var}[\Lambda_T] \pm O(n^{-1/9})\right)\\ &= \exp\left(-\Lambda_2 + \Lambda_4 \pm O(n^{-1/9})\right).
	\end{align*}
	
	Plugging this last estimate into \cref{prop:graph-expectation}, we obtain
\[	\mb{P}(\deg_S(v) = t)
		= (1\pm O(n^{-1/10}))\frac{\binom{h}{t}\binom{n-h-1}{d_v-t}}{\binom{n-1}{d_v}}\exp(\Lambda_1 - \Lambda_2 - \Lambda_3 + \Lambda_4),
\]
as desired, using the fact that the product of binomials in question changes by a small factor depending on whether $v\in S$ or $v\in S^c$, a factor which is nonetheless subsumed by the error term with room to spare.
\end{proof}

The proof of \cref{lem:enum-consequence-bipartite} is analogous to the argument above, so in this case, we only record the appropriate intermediate results needed, and omit the details.

\cref{lem:enum-consequence-bipartite} may be deduced from the following result in the same fashion as \cref{lem:enum-consequence} was from \cref{prop:graph-bounded}.

\begin{proposition}\label{prop:bigraph-bounded}
Let $((d_v)_{v\in V}, (d_w)_{w\in W})$ be a pair of sequences with identical sums on a bipartition $V\cup W$ with $|V|$ = m, $|W|=n$ such that, defining $\alpha_v$ by $d_v = (n-1)/2+\alpha_v\sqrt{(n-1)} / 2$ for $v \in V$ and  $\beta_w$ by $d_w = (n-1)/2+\beta_w\sqrt{(n-1)} / 2$ for $w \in W$, we have
\begin{itemize}
\item $(\log n)^{-1/4}\le m/n\le(\log n)^{1/4}$,
\item $|\alpha_v| \le \log n$ for each $v \in V$ and $|\beta_w| \le \log n$ for each $w \in W$, and
\item (n/m) $\sum_{w \in W} \beta_w^2 \le n (\log n)^{1/9}$.
\end{itemize}
Then it is a bipartite-graphic degree sequence (for $n$ large). Let $G$ be a uniformly random graph with this degree sequence on the vertex set $W$. Consider any fixed $u \in V$, any fixed subset $S\subseteq W$ of size $h$ satisfying $\min(h,n-h)\ge n/(\log n)^{1/8}$, and an integer $t\in[0,d_u]$. If $|t-h/2| > n^{3/5}$, then we have
	\begin{equation}\label{bigr-bd-big}
	    \mb{P}(\deg_S(u) = t)  \le \exp(-\Omega((t-h/2)^2/n)).
	\end{equation}
If $|t-h/2|\le n^{3/5}$ on the other hand, then we have
        \begin{equation}\label{bigr-bd-small}
        \mb{P}(\deg_S(u) = t)
		= (1\pm O(n^{-1/10}))\frac{\binom{h}{t}\binom{n-h}{d_u-t}}{\binom{n}{d_u}}\exp(\Lambda_1 - \Lambda_2 - \Lambda_3 + \Lambda_4),
		\end{equation}
where $\Lambda_1$, $\Lambda_2$, $\Lambda_3$ and $\Lambda_4$ are given by	
\begin{align*}
		\Lambda_1 &= \frac{1}{2mn}\left(\sum_{i \in W}\beta_i\right)\left(\sum_{i\in W}\beta_i-2\sqrt{mn}\alpha_u\right),\\
		\Lambda_2 &= \sum_{i\in S}\left(1-\frac{2t}{h}\right)\frac{\beta_i}{\sqrt{m}}+\sum_{i\in W\setminus S}\left(1-\frac{2(d_v-t)}{(n-h)}\right)\frac{\beta_i}{\sqrt{m}},\\
		\Lambda_3 &= \frac{1}{2}\sum_{i\in W}\frac{\beta_i^2}{m}, \text{ and }\\
		\Lambda_4 &= \frac{1}{2mh}\sum_{i, j\in S}(\beta_i-\beta_j)^2 + \frac{1}{2m(n-h)}\sum_{i ,j\in W\setminus S}(\beta_i-\beta_j)^2,
	\end{align*}
the sums in the definition of $\Lambda_4$ being over all (unordered) two-element subsets. \qed
\end{proposition}

As before, \cref{prop:bigraph-bounded} itself is a consequence of the following result, whose proof will be sketched in \cref{app:Proof} once we have collected the requisite machinery in \cref{app:enumresults}.

\begin{proposition}\label{prop:bigraph-expectation}
Let $((d_v)_{v\in V}, (d_w)_{w\in W})$ be a pair of sequences of identical sums on a bipartition $V\cup W$ with $|V|$ = m, $|W|=n$ such that, defining $\alpha_v$ by $d_v = (n-1)/2+\alpha_v\sqrt{(n-1)} / 2$ for $v \in V$ and  $\beta_w$ by $d_w = (n-1)/2+\beta_w\sqrt{(n-1)} / 2$ for $w \in W$, we have
\begin{itemize}
\item $(\log n)^{-1/4}\le m/n\le(\log n)^{1/4}$,
\item $|\alpha_v| \le \log n$ for each $v \in V$ and $|\beta_w| \le \log n$ for each $w \in W$.
\end{itemize}
Such a pair of sequences form a bipartite-graphic sequence for all sufficiently large $n$. Let $G$ be a uniformly random graph with this degree sequence on the vertex set $W$. For any fixed $u \in V$, $S\subseteq W$ of size $h$ satisfying $\min(h,n-h)\ge n/(\log n)^{1/8}$, and an integer $t\in[0,d_u]$, we have
	\[
		\mb{P}(\deg_S(u) = t)=(1\pm O(n^{-1/8}))\frac{\binom{h}{t}\binom{n-h}{d_u-t}}{\binom{n}{d_u}} \exp(\Lambda_1 - \Lambda_3)
		\mb{E}_T\left[\exp(-\Lambda_T)\right],
	\]
	where $T = T_1 \cup T_2$ is a random set chosen by picking $T_1$ uniformly from $\binom{S}{t}$ and $T_2$ uniformly from $\binom{W\setminus S}{d_u-t}$, and where $\Lambda_1$, $\Lambda_3$ and $\Lambda_T$ are given by	
\begin{align*}
		\Lambda_1 &= \frac{1}{2mn}\left(\sum_{i \in W}\beta_i\right)\left(\sum_{i\in W}\beta_i-2\sqrt{mn}\alpha_u\right),\\
		\Lambda_3 &= \frac{1}{2}\sum_{i\in W}\frac{\beta_i^2}{m}, \text{ and }\\
    	\Lambda_T &= \sum_{i\in W}(-1)^{\mbm{1}_T(i)}\frac{\beta_i}{\sqrt{m}}.
	\end{align*}
\end{proposition}

\section{Graph enumeration results and related estimates}\label{app:enumresults}
The main tools needed to prove \cref{prop:graph-expectation,prop:bigraph-expectation} are the following enumeration theorems of McKay and Wormald~\cite{MW90} and of Canfield, Greenhill, and McKay~\cite{CGM08}. 

\begin{theorem}\label{thm:enum-graph}
There exists a fixed constant $\varepsilon>0$ such that the following holds. Consider a sequence $\mbf{d} = (d_1, \dots, d_n)$ with even sum such that, writing $\ol{d}=(1/n)\sum_{i=1}^n d_i$, we have
\begin{itemize}
    \item $|d_i-\ol{d}|\le n^{1/2+\varepsilon}$ for $1 \le i \le n$, and
    \item $\ol{d}\ge n/\log n$.
\end{itemize}
Writing $m = \ol{d}n/2\in\mb{Z}$, $\mu = \ol{d}/(n-1)$, and $\gamma_2^2 = (1/(n-1)^2)\sum_{i=1}^n(d_i-\ol{d})^2$, the number of labelled graphs with degree sequence $\mbf{d}$ is
\begin{align*}
    (1\pm O(n^{-1/4}&))  \exp\left(\frac{1}{4}-\frac{\gamma_2^2}{4\mu^2(1-\mu)^2}\right)\binom{n(n-1)/2}{m}\binom{n(n-1)}{2m}^{-1}\prod_{i=1}^n\binom{n-1}{d_i}. \tag*\qed
\end{align*}
\end{theorem}

\begin{theorem}\label{thm:enum-bigraph}
There exists a fixed constant $\varepsilon>0$ such that the following holds. 
Consider a pair of sequences $(\mbf{s} = (s_1, \dots, s_n), \mbf{t} = (t_1, \dots, t_m))$ with identical sums such that, writing $\ol{s}=(1/n)\sum_{i=1}^n s_i$ and $\ol{t}=(1/n)\sum_{i=1}^m t_i$, we have

\begin{itemize}
    \item $n/(\log n)^{1/2}\le m\le n(\log n)^{1/2}$,
    \item $|s_i-\ol{s}|\le n^{1/2+\varepsilon}$ for $1 \le i \le n$ and $|t_i-\ol{t}| \le m^{1/2+\eps}$ for $1 \le i \le m$, and
    \item $\ol{s}\ge n/(\log n)^{1/2}$ and $\ol{t} \ge m/(\log m)^{1/2}$.
\end{itemize}

Writing $\mu = \sum_{i=1}^n s_i/(mn) = \sum_{i=1}^m t_i/(mn)$, $\gamma_2(\mbf s)^2 = (1/n^2)\sum_{i=1}^n(s_i-\ol{s})^2$ and $\gamma_2(\mbf t)^2 = (1/m^2)\sum_{i=1}^m(t_i-\ol{t})^2$, the number of labelled bipartite graphs whose partition classes have degree sequences $\mbf{s}$ and $\mbf{t}$ is
\begin{align*}
(1\pm O(n^{-1/8}&))\exp\left(-\frac{1}{2}\left(1-\frac{\gamma_2(\mbf{s})^2}{\mu(1-\mu)}\right)\left(1-\frac{\gamma_2(\mbf{t})^2}{\mu(1-\mu)}\right)\right)\binom{mn}{mn\mu}^{-1}\prod_{i=1}^n\binom{m}{s_i}\prod_{i=1}^m\binom{n}{t_i}.\tag*\qed
\end{align*}
\end{theorem}

We remark that these enumeration results are now known to hold under even broader conditions on the degree sequences (i.e., $\mbf{d}$, $\mbf{s}$ and $\mbf{t}$) due to works of Barvinok and Hartigan~\cite{BH13}, and for essentially all sparsities by recent work of Liebenau and Wormald~\cite{LW17,LW20}. We refer the reader to~\cite{Wor18} for an excellent survey of these results.

In order to estimate the expressions in \cref{thm:enum-graph,thm:enum-bigraph}, we shall also require the following estimates for binomial coefficients. These follow from sufficiently precise versions of Stirling's approximation for the factorial. These estimates are nonetheless somewhat nonstandard, and so we include proofs, following the exceptionally clean approach in~\cite{Spe14}.

\begin{lemma}\label{lem:binomial-approximation}
We have the following pair of estimates.
\begin{enumerate}
    \item Let $\Delta_1 = e-m(m-1)/4$ and $\Delta_2 = (m-1)/2-d$. If $|\Delta_1| = O(m^{3/2})$ and $|\Delta_2| = O(\sqrt{m\log m})$, then
    \[\frac{\binom{m(m-1)/2}{e}\binom{m(m-1)}{2e}^{-1}}{\binom{(m-1)(m-2)/2}{e-d}\binom{(m-1)(m-2)}{2e-2d}^{-1}} = (1\pm O(m^{-2/5})) 2^{-(m-1)}\exp(-8(\Delta_1^2+\Delta_1\Delta_2m)/m^3).\]
    \item Let $\Delta_1 = e-mn/2$ and $\Delta_2 = n/2-d$. If $|\Delta_1|\le O(m^{3/2})$, $|\Delta_2|\le O(\sqrt{m\log m})$, and $m=\Theta(n)$, then
    \[\binom{mn}{e}^{-1}\binom{(m-1)n}{e-d} = (1\pm O(m^{-2/5})) 2^{-n}\exp(-2(2m\Delta_1\Delta_2+\Delta_1^2)/(m^2n)).\]
\end{enumerate}
\end{lemma}
\begin{proof}
We first compute $\binom{n}{(n+i)/2}$ to sufficient precision when $|i|\le n^{4/5}$. Note that 
\[\binom{n}{(n+i)/2}\binom{n}{n/2}^{-1} = \prod_{j=1}^{i/2}\frac{n/2-j+1}{n/2+j} =\prod_{j=1}^{i/2}\frac{n/2-j}{n/2+j}\prod_{j=1}^{i/2}\frac{n/2-j+1}{n/2-j}.\]
The final product above is $(1 \pm O(n^{-1/5}))$ and may be safely ignored. For the first of the two products, note that 
\begin{align*}
\sum_{j=1}^{i/2}\log((n/2-j)/(n/2+j)) &= \sum_{j=1}^{i/2}-4j/n-2(2j/n)^3/3  \pm O(n^{-1/5})\\
&= -i^2/(2n) - i^4/(12n^3) \pm O(n^{-1/5}).
\end{align*}

Now, we have $\Delta_1 = e-m(m-1)/2$ and $\Delta_2 = (m-1)/2-d$. Applying this to the first ratio of binomial coefficients, we find that \begin{align*}
\frac{\binom{m(m-1)/2}{e}\binom{m(m-1)}{2e}^{-1}}{\binom{(m-1)(m-2)/2}{e-d}\binom{(m-1)(m-2)}{2e-2d}^{-1}} &= (1\pm O(m^{-2/5})) 2^{-(m-1)}\exp(-8(\Delta_1^2+\Delta_1\Delta_2m)/m^3),
\end{align*}
proving the first estimate. Next, note that $\Delta_1 = e-mn/2$ and $\Delta_2 = n/2-d$, so
\begin{align*}
\binom{mn}{e}^{-1}\binom{(m-1)n}{e-d} = (1\pm O(m^{-2/5})) 2^{-n}\exp(-2(2m\Delta_1\Delta_2+\Delta_1^2)/(m^2n)),
\end{align*}
proving the second estimate.
\end{proof}

\section{Proofs of the main technical estimates}\label{app:Proof}
With the results in \cref{app:enumresults} in hand, we are now ready to prove \cref{prop:graph-expectation,prop:bigraph-expectation}. We start with \cref{prop:graph-expectation}.
\begin{proof}[Proof of \cref{prop:graph-expectation}]
Given $\mbf{d} = (d_w)_{w \in W}$, $v \in W$ and $T\subseteq W\setminus v$ of size $d_v$, we shall estimate the probability of the neighbourhood of $v$ in $G$ being exactly $T$. 

To this end, let $\mbf{d}_T = (d_w-\mbm{1}_{T}(w))_{w \in W}$. As in \cref{thm:enum-graph}, let
\begin{align*}
\ol{d} &= \frac{1}{n}\sum_{i\in W}d_i,&&\ol{d}_T= \frac{1}{n-1}\sum_{i\in W\setminus v}d_i-\mbm{1}_{T}(i) = \frac{n\ol{d}-2d_v}{n-1},\\
r &= \frac{\ol{d}n}{2},&&r_T = \frac{\ol{d}_T(n-1)}{2} = r-d_v,\\
\mu &= \frac{\ol{d}}{n-1},&&\mu_T = \frac{\ol{d}_T}{n-2} = \frac{n}{n-2}\mu-\frac{2d_v}{(n-1)(n-2)},\\
\gamma_2^2 &= \frac{1}{(n-1)^2}\sum_{i\in W}(d_i-\ol{d})^2,&&\gamma_2^2(T) = \frac{1}{(n-2)^2}\sum_{i\in W\setminus v}(d_{T,i}-\ol{d}_T)^2.
\end{align*}

Note that $\mbf{d}$ and $\mbf{d}_T$ both clearly satisfy the conditions of \cref{thm:enum-graph} due to our hypotheses, and that
\[\gamma_2^2(T) = \gamma_2^2 \pm O(n^{-1/4}) \text{ and } \mu_T = \mu \pm O(1/n),\]
again, from the given hypotheses. Now define
\[ \Phi = \frac{\binom{(n-1)(n-2)/2}{r-d_v}\binom{(n-1)(n-2)}{2r-2d_v}^{-1}}{\binom{n(n-1)/2}{r}\binom{n(n-1)}{2r}^{-1}}2^{-(n-1)}\]
and recall $d_i = (n-1)/2+\beta_i\sqrt{(n-1)} / 2$. We have
\[r-\frac12\binom{n}{2} = \frac{1}{2}\sum_{i\in W}(d_i-(n-1)/2) = \frac{\sqrt{(n-1)}}{4}\sum_{i\in W}\beta_i.\]
From our hypotheses and the first estimate in \cref{lem:binomial-approximation}, we then deduce that
\begin{align*}
\Phi &= \exp\left(\frac{(\sum_{i\in W}\beta_i)(\sum_{i\in W}\beta_i-2n\beta_n)}{2n^2} \pm O(n^{-1/6})\right)\\
&= \exp(\Lambda_1 \pm O(n^{-1/6})).    
\end{align*}

The above estimates for $\gamma_2^2(T)$ and $\mu_T$ imply that
\[\frac{\exp\left(\frac{1}{4}-\frac{\gamma_2^2(T)}{4\mu_T^2(1-\mu_T)^2}\right)}{\exp\left(\frac{1}{4}-\frac{\gamma_2^2}{4\mu^2(1-\mu)^2}\right)} = 1 \pm O(n^{-1/4}),\]
and this fact in conjunction with \cref{thm:enum-graph} yields
\begin{align*}
\mb{P}[N(v) = T]&= (1\pm O(n^{-1/4}))\frac{\binom{(n-1)(n-2)/2}{r_T}\binom{(n-1)(n-2)}{2r_T}^{-1}\prod_{i \in W \setminus v}\binom{n-2}{d_i-\mbm{1}_T(i)}}{\binom{n(n-1)/2}{r}\binom{n(n-1)}{2r}^{-1}\prod_{i \in W}\binom{n-1}{d_i}}\\
&= (1\pm O(n^{-1/4}))\frac{\Phi 2^{n-1}}{\binom{n-1}{d_v}}\prod_{i\in T}\frac{d_i}{n-1}\prod_{i\notin T}\frac{n-1-d_i}{n-1}\\
&= (1\pm O(n^{-1/4}))\frac{\Phi}{\binom{n-1}{d_v}}\prod_{i\in T}\left(1+\frac{\beta_i}{\sqrt{n-1}}\right)\prod_{i\notin T}\left(1-\frac{\beta_i}{\sqrt{n-1}}\right)\\
&= \frac{\Phi}{\binom{n-1}{d_v}}\exp\left(-\sum_{i\in W\setminus v}(-1)^{\mbm{1}_T(i)}\frac{\beta_i}{\sqrt{n-1}}-\frac{1}{2}\sum_{i\in W\setminus v}\frac{\beta_i^2}{n-1}\pm O(n^{-1/4})\right)\\
&= \frac{\Phi}{\binom{n-1}{d_v}}\exp\left(- \Lambda_T - \Lambda_3 \pm O(n^{-1/4})\right).
\end{align*}

Since the above estimate holds for every choice of $T \subseteq W \setminus v$, we may finish by noting that
\begin{align*}
(1\pm O(n^{-1/4}))\frac{\binom{n-1}{d_v}}{\Phi \binom{h-\mbm{1}_S(v)}{t}\binom{n-h-\mbm{1}_{S^c}(v)}{d_v-t}}\mb{P}[\deg_S(v) = t]= \exp(-\Lambda_3) \mb{E}_{T}[\exp(-\Lambda_T)],
\end{align*}
where $T = T_1 \cup T_2$ is a random set chosen by picking $T_1$ uniformly from $\binom{S}{t}$ and $T_2$ uniformly from $\binom{W\setminus S}{d_v-t}$. Rearranging this, and recalling that $\Phi = \exp(\Lambda_1 \pm O(n^{-1/6}))$, gives us the desired result.
\end{proof}

To finish, we outline the proof of \cref{prop:bigraph-expectation}.

\begin{proof}[Proof of \cref{prop:bigraph-expectation}]
The proof of this proposition mirrors that of \cref{prop:graph-expectation}, except now using \cref{thm:enum-bigraph} instead of \cref{thm:enum-graph}, and the second estimate in \cref{lem:binomial-approximation} instead of the first. Since the requisite calculations are routine (and are analogous to those spelled out in the proof \cref{prop:graph-expectation}), we leave the details of these calculations to the reader.
\end{proof}

\end{document}